\documentclass[a4paper,reqno,11pt]{amsart}
\usepackage{
amssymb,
amsmath,
amsthm,
eucal,
dsfont,
multicol,
mathrsfs,color,
hyperref
}

\makeatletter\renewcommand*{\eqref}[1]{%
\hyperref[{#1}]{\textup{\tagform@{\!\!\ref*{#1}}}}%
}\makeatother %add when putting this file to arXiv

\makeatletter
 
  \@addtoreset{equation}{section}
 \makeatother
\theoremstyle{plain}

\newtheorem{theorem}{Theorem}[section]
\newtheorem{lemma}[theorem]{Lemma}

\newtheorem{corollary}[theorem]{Corollary}
\theoremstyle{definition}

\newtheorem{remark}[theorem]{Remark}

\newtheorem{assumption}[theorem]{Assumption}
\newtheorem{theoremA}{Theorem}
\newcommand{\bignorm}[1]{{\left\|#1\right\|}}
\newcommand{\norm}[1]{{\|#1\|}}

\def\supp{\mathop{\mathrm{supp}}\nolimits}

\def\Im{\mathop{\mathrm{Im}}\nolimits}
\def\loc{\mathop{\mathrm{loc}}\nolimits}

\def\R{{\mathbb{R}}}

\def\N{{\mathbb{N}}}
\def\C{{\mathbb{C}}}

\def\H{{\mathcal{H}}}

\def\X{{\mathcal{X}}}
\def\Y{{\mathcal{Y}}}

\def\<{{\langle}}
\def\>{{\rangle}}

\def\ep{{\varepsilon}}
\def\ds{\displaystyle}
\DeclareMathOperator*{\slim}{s-lim}

\title[Scattering theory in homogeneous Sobolev spaces]{Scattering theory in homogeneous Sobolev spaces for Schr\"odinger and wave equations with rough potentials}

\author{Haruya Mizutani}
\address{Department of Mathematics, Graduate School of Science, Osaka University, Toyonaka, Osaka 560-0043, Japan.}
\email{haruya@math.sci.osaka-u.ac.jp}

\begin{document}
%\date{\today}

\begin{abstract}
We study the scattering theory for the Schr\"odinger and wave equations with rough potentials in a scale of homogeneous Sobolev spaces. The first half of the paper concerns with an inverse-square potential in both of subcritical and critical constant cases, which is a particular model of scaling-critical singular perturbations. In the subcritical case, the existence of the wave and inverse wave operators defined on a range of homogeneous Sobolev spaces is obtained. In particular, we have the scattering to a free solution in the homogeneous energy space for both of the Schr\"odinger and wave equations. In the critical case, it is shown that the solution is asymptotically a sum of a $n$-dimensional free wave and a rescaled two-dimensional free wave. The second half of the paper is concerned with a generalization to a class of strongly singular decaying potentials. We provides a simple criterion in an abstract framework to deduce the existence of wave operators defined on a homogeneous Sobolev space from the existence of the standard ones defined on a base Hilbert space. 
\end{abstract}
\maketitle

\section{Introduction}
The present paper is mainly concerned with the scattering theory for the following Schr\"odinger and wave equations in homogeneous Sobolev spaces:
\begin{align*}
i\partial_tu+\Delta u-V(x)u&=0,\ u|_{t=0}\in \dot H^s;\\ \partial_t^2v-\Delta v+V(x)v&=0,\ (v,\partial_t v)|_{t=0}\in \dot H^s\times \dot H^{s-1},
\end{align*}
where $\dot H^s=\dot H^s(\R^n)$ denotes the homogeneous Sobolev space of order $s$. 
The purpose of the present work is twofold. The first half is concerned with the inverse-square potential $V(x)=a|x|^{-2}$ with $a\ge-(n-2)^2/4$ and $n\ge3$, which can be regarded as a particular example of scaling-critical singular potentials. In the subcritical case $a>-(n-2)^2/4$, we prove the existence of the wave and inverse wave operators defined on $\dot H^{s}$ for the Schr\"odinger equation or on $\dot H^s\times \dot H^{s-1}$ for the wave equation, respectively, for a range of $s$ including the case $s=1$. In particular, we obtain the scattering to a free solution in the homogeneous energy space. In the critical case $a=-(n-2)^2/4$, we also study the asymptotic behaviors of the solutions which are  different from the subcritical case.  The second half of the paper is devoted to a generalization of the first half part to an abstract framework. Our method can be applied to a wide class of strongly singular potentials, including a rough potential $V\in L^{n/2}(\R^n)$ with a small negative part. 

Let us briefly explain the main results of the paper, describing some motivation from nonlinear scattering theory. Since a seminal work by Burq et al \cite{BPST1,BPST2} on the Strichartz estimates, the following nonlinear Schr\"odinger (NLS) and nonlinear wave (NLW)  equations with an inverse-square potential have attracted increasing attention from a scattering theory viewpoint (see \cite{BPST1,ZhZh,KMVZZ,LMM,Suzuki_EECT,Yang} and reference therein for \eqref{NLS} and \cite{MMZ,MSZ} for \eqref{NLW}, respectively): 
\begin{align}
\label{NLS}
i\partial_t u-H_au&=\lambda |u|^{p-1}u,\quad u|_{t=0}=u_0;\\
\label{NLW}
\partial_t^2 v+H_av&=-\lambda |v|^{p-1}v,\quad (v,\partial_t v)|_{t=0}=(v_0,v_1),
\end{align}
where $H_a=-\Delta+a|x|^{-2}$, $u:\R\times\R^n \to\C$ and $v:\R\times\R^n\to \R$. In particular, the following results on the nonlinear scattering  in the energy critical case were proved in \cite{KMVZZ,MMZ}.  For simplicity, we state the defocusing case only (see Remark \ref{remark_1_2} on the focusing case). 

%proposition
\begin{theoremA}[\cite{KMVZZ} for \eqref{NLS} and  \cite{MMZ} for \eqref{NLW}]
\label{theorem_A}
{\it Let $\lambda=1$, $p=1+4/(n-2)$ and suppose that
$$
n=\begin{cases}3&\text{for \eqref{NLS},}\\3\ \text{or}\ 4&\text{for \eqref{NLW},}\end{cases}\quad a>-\frac{(n-2)^2}{4}+c_n,\quad c_n=\begin{cases}\frac{1}{25}&\text{if}\ n=3,\\\frac19&\text{if}\ n=4.\end{cases}
$$
Then the following statements are satisfied:
\begin{itemize}
\item For any $u_0\in \dot H^1$, \eqref{NLS} has a unique global solution $u\in C(\R;\dot H^1)$ which {\it scatters to a linear solution} in $\dot H^1$ in the sense that there exist unique $u_\pm\in \dot H^1$ such that 
$$
\lim_{t\to\pm\infty}\norm{u(t)-e^{-itH_a}u_\pm}_{\dot H^1}=0.
$$
\item For any $\vec v_0=(v_0,v_1)\in \dot H^1\times L^2$, \eqref{NLW} has a unique global solution $\vec{v}=(v,\partial_tv)\in C(\R;\dot H^1\times L^2)$ which scatters to a linear solution in $\dot H^1\times L^2$, i.e., there exist unique $\vec{v}_\pm\in \dot H^1\times L^2$ such that
$$
\lim_{t\to\pm\infty}\norm{\vec v(t)-S_a(t)\vec v_\pm}_{\dot H^1\times L^2}=0.
$$
\end{itemize}}
\end{theoremA}
Here $S_a(t)$ denotes the evolution group for the linear wave equation associated with $H_a$ (see \eqref{S_a(t)} below). This theorem shows that the global dynamics of \eqref{NLS} and \eqref{NLW} are asymptotically governed by the linear evolution groups $e^{-itH_a}$ and $S_a(t)$, respectively.  Note that $e^{-itH_a}$ and $S_a(t)$ still depend on the linear potential $a|x|^{-2}$. On the other hand, by employing for instance the Strichartz estimates proved in \cite{BPST1}, one can prove {\it the scattering to a free solution} for $e^{-itH_a}$ in $L^2$ and for $S_a(t)$ in $\dot H^{1/2}\times\dot H^{-1/2}$, respectively: if $n\ge3$, $a>-(n-2)^2/4$,  $u_0\in L^2$ and $\vec v_0\in \dot H^{1/2}\times\dot H^{-1/2}$, then there exist unique $u_\pm \in L^2$ and $\vec v_\pm\in \dot H^{1/2}\times\dot H^{-1/2}$ such that 
\begin{align}
\label{eqn_introduction_3}
\lim_{t\to\pm\infty}\norm{e^{-itH_a}u_0-e^{-itH_0}u_\pm}_{L^2}=0,\\
\label{eqn_introduction_4}
\lim_{t\to\pm\infty}\norm{S_a(t)\vec v_0-S_0(t)\vec v_\pm}_{\dot H^{1/2}\times \dot H^{-1/2}}=0,
\end{align}
respectively. However, these nonlinear and linear scattering results are not sufficient to ensure the scattering to a free solution for the solution to \eqref{NLS} or \eqref{NLW} since there is no embeddings between two different homogeneous Sobolev spaces. There seems to be no previous literature, except for \cite{MSZ}, on the scattering to a free solution for  \eqref{NLS} or \eqref{NLW} if $a,\lambda\neq0$. Recently, the authors of \cite{MSZ} showed \eqref{eqn_introduction_4} with $\dot H^{1/2}\times \dot H^{-1/2}$ replaced by $\dot H^{1}\times L^2$ for radial initial data, which implies the scattering to a free solution in $\dot H^{1}\times L^2$ for the solution to \eqref{NLW} under the radial symmetry and conditions in Theorem \ref{theorem_A}. 

In the first half part, we show the scattering to a free solution in the homogeneous energy space without the radial symmetry, namely \eqref{eqn_introduction_3} and \eqref{eqn_introduction_4} with $L^2$ and $\dot H^{1/2}\times \dot H^{-1/2}$ replaced by $\dot H^1$ and $\dot H^1\times L^2$, respectively. Together with Theorem \ref{theorem_A}, this leads the following result: 

\begin{corollary}
\label{corollary_1_1}
Let $u(t)$ and $\vec v(t)$ be the global solutions to \eqref{NLS} and \eqref{NLW}, respectively,  obtained by Theorem \ref{theorem_A}. 
Then, there exist $u_\pm\in \dot H^1$ such that $u(t)=e^{-itH_0}u_\pm+o(1)$ in $\dot H^1$ as $t\to \pm\infty$. Similarly, there exist $\vec v_\pm\in \dot H^1\times L^2$ such that $\vec v(t)=S_0(t)\vec v_\pm+o(1)$ in $\dot H^1\times L^2$ as $t\to \pm\infty$. 
\end{corollary}

%remark
\begin{remark}
\label{remark_1_2}
We here give a few remarks on Corollary \ref{corollary_1_1}. 
\begin{itemize}
\item[(1)] In fact we prove that \eqref{eqn_introduction_3} and \eqref{eqn_introduction_4} hold with $L^2$ and $\dot H^{1/2}\times \dot H^{-1/2}$ replaced by $\dot H^s$ and $\dot H^s\times \dot H^{s-1}$, respectively, for a range of $s$ including the cases $s=1,1/2$. 
\item[(2)] The scattering to a linear solution in the homogeneous energy space for the focusing ($\lambda=-1$), energy critical case with the energy less than that of the ground state has been also studied by \cite{Yang} for \eqref{NLS} (under  the radial symmetry) and \cite{MMZ} for \eqref{NLW}, respectively. For these models, one can also obtain the scattering to a free solution by the same argument as above. 
\end{itemize}
\end{remark}

In the first half part, the critical case $a=-(n-2)^2/4$ is also studied. We show in the critical case that the radial part of $e^{-itH_a}u_0$ coincides with a rescaled two-dimensional free solution $U^*e^{it\Delta_{\R^2}}Uu_0$ where $U$ is a unitary from $L^2_{\mathrm{rad}}(\R^n)$ to $L^2_{\mathrm{rad}}(\R^2)$, while the other part scatters to a free solution $e^{-itH_0}u_\pm$. A similar property for $S_a(t)$ is also proved. This shows that the global behavior for $e^{-itH_a}$ or $S_a(t)$ in the critical case is different from the subcritical case. 

The argument in the first half part uses several specific properties of $H_a$. Among others, the equivalence between the homogeneous Sobolev norm $\norm{H_a^{s/2}f}_{L^2}$ associated with $H_a$ and the standard one $\norm{f}_{\dot H^s}$ established by Killip et al \cite{KMVZZ2} plays an important role. More precisely, a crucial point to prove the scattering in $\dot H^s$ is that such a norm equivalence holds for some $s'$ satisfying $s<s'$. A similar argument is also used in case of the wave equation. However, it seems to the author that this condition is too strong, especially when generalizing to a wide class of strongly singular potentials or to an abstract framework. 
In the second half part, an alternative approach to ensure the scattering in $\dot H^s$ is established. We provide a simple criterion in an abstract framework to deduce the existence of the wave operators defined on a homogeneous Sobolev space from the existence of the standard ones defined on a base Hilbert space. This criterion requires the norm equivalence only up to $\dot H^s$ and a few additional assumptions. For instance, in case of the Schr\"odinger equation $i\partial_t u-H_0u-V(x)u=0$, the existence of the wave and inverse wave operators, 
\begin{align*}
\slim_{t\to\pm\infty}e^{it(H_0+V)}e^{-itH_0},\quad
\slim_{t\to\pm\infty}e^{itH_0}e^{-it(H_0+V)}P_{\mathrm{ac}}(H_0+V),
\end{align*}
defined as the strong limits on $\dot H^1$, can be deduced from the existence of the standard ones defined on $L^2$ (which is well understood), the norm equivalence between $\norm{f}_{\dot H^1}$ and $\norm{(H_0+V)^{1/2}f}_{L^2}$ and the compactness on $L^2$ of the operator
$$
\<H_0\>^{1/2}\left\{(H_0+i)^{-1}-(H_0+V+i)^{-1}\right\}=\<H_0\>^{1/2}(H_0+V+i)^{-1}V(H_0+i)^{-1}.
$$
The last condition is much easier to prove than the norm equivalence beyond $\dot H^1$. For instance, it follows from the $H_0$-form compactness of $V$. Moreover, our criterion can be applied to not only a pair of the Schr\"odinger evolution groups $(e^{-itH_0},e^{-it(H_0+V)})$, but also $(e^{-itf(H_0)},e^{-itf(H_0+V)})$ for a class of functions $f$, which particularly includes the case $f(\lambda)=\lambda^\alpha$ with $\alpha>0$. Hence it can be applied to prove the existence of wave operators on homogeneous Sobolev spaces for the half wave equations $i\partial_t u\pm\sqrt{H_0+V}u=0$, as well as the wave equation $\partial_t^2v+H_0v+Vv=0$. The second half part is a continuation of our previous work \cite{Miz_PAMS} in which the scattering theory for Schr\"odinger equations has been studied in a scale of inhomogeneous Sobolev spaces. 

The mathematical study of linear scattering theory, especially the wave operators, has a long history and there is a huge amount of literatures. We only refer to some monographs \cite{ReSi_III,ReSi_IV,Yafaev_I,Yafaev_II} for the Schr\"odinger equation and \cite{LaPh} for the wave equation. There are also many works on the scattering to a linear solution (not a free solution) for the nonlinear NLS equation with a rough linear potential having a small negative part (see, in addition to aforementioned papers, \cite{Hong,BaVi,Laf,IkIn} and references therein). The long time asymptotics of small solutions to  the NLS equations with a localized large linear potential has been also extensively studied (see {\it e.g.}, \cite{SoWe,TsYa,GNT,Mizumachi,CuMa,Nakanishi} and references therein). It however seems to the author that the linear scattering theory in homogeneous Sobolev spaces has attracted less interest except for the wave equation in the energy space. For the linear scattering theory in inhomogeneous Sobolev spaces, in addition to our previous work \cite{Miz_PAMS}, we refer to \cite{FoVi,TzVi}. %Besides the case with the inverse-square potential, we believe that the results in the present paper have several potential applications to the study of scattering theory for NLS and NLW equations with a linear potential. 

\subsection*{Organization of the paper} Section \ref{section_2} is concerned with the scattering theory with the inverse-square potential. The abstract criterion is given in Section \ref{section_3}. Some applications to the Schr\"odinger and wave equations with rough potentials are studied in \ref{section_4}. In Appendix \ref{appendix_A}, we record the invariance principle of the wave operators which will play an important role in case of the wave equation. 

\subsection*{Notation} The following standard notations are used in the paper: 
\begin{itemize}
\item $H^s=H^s(\R^n)$ and $\dot H^s=\dot H^s(\R^n)$ denote the $L^2$-based inhomogeneous and homogeneous Sobolev spaces of order $s$, respectively;
\item $\mathbb B (\X,\Y)$ (resp. $\mathbb B_\infty (\X,\Y)$) denotes the family of bounded (resp. compact) operators from $\X$ to $\Y$. We also set $\mathbb B (\X)=\mathbb B (\X,\X)$ and $\mathbb B_\infty(\X)=\mathbb B_\infty(\X,\X)$. 
\item Given a self-adjoint operator $A$, $P_{\mathrm{ac}}(A)$ denotes the projection onto the absolutely continuous subspace of $A$. Moreover $E_A(\cdot)$ is the spectral measure associated with $A$; 
\item $\sigma(A),\rho(A)\subset \C$ denote the spectrum and the resolvent set for an operator $A$, respectively. 
\end{itemize}

\section{Scattering theory for the inverse-square potential}
\label{section_2}
This section is devoted to a scattering theory for the Schr\"odinger and wave equations with the inverse-square potential. Throughout the section we always assume $n\ge3$. Let $$H_a:=-\Delta+a|x|^{-2}$$ be the Schr\"odinger operator with the inverse-square potential on $\R^n$, where $\Delta=\sum_{j=1}^n\partial_j^2$ is the Laplacian. To be more precise, thanks to the sharp Hardy inequality
\begin{align}
\label{Hardy}
\frac{(n-2)^2}{4}\int |x|^{-2}|u|^2dx\le \int |\nabla u|^2dx,\quad u\in C_0^\infty(\R^n),
\end{align}
 the quadratic form 
$$
Q_{a}(u)=\int \left(|\nabla u|^2+a|x|^{-2}|u|^2\right)dx,\quad u\in C_0^\infty(\R^n),
$$
is non-negative and closable if $a\ge -(n-2)^2/4$. Then $H_a$ is defined as its Friedrichs extension, that is a unique self-adjoint operator generated by the closure of $Q_{a}$. Since $H_a$ is non-negative, one can define its square root $|D_a|=H_a^{1/2}$ via the spectral theorem, namely
$$
|D_a|=\int_0^\infty \lambda^{1/2}dE_{H_a}(\lambda)
$$
where $dE_{H_a}$ is the spectral measure associated with $H_a$. By Hardy's inequality \eqref{Hardy}, 
\begin{align}
\label{2_2}
Q_a(u)\le C\norm{u}_{\dot H^1}^2,\quad u\in C_0^\infty(\R^n)
\end{align}
for all $a\ge -(n-2)^2/4$. Moreover, one has
\begin{align}
\label{2_3}
\norm{u}_{\dot H^1}^2\le CQ_a(u),\quad u\in C_0^\infty(\R^n)
\end{align}
if and only if $a> -(n-2)^2/4$. In particular, the form domain $D(|D_a|)$ of $H_a$ coincides with $H^1$ if and only if $a>-(n-2)^2/4$. When $a=-(n-2)^2/4$, $D(|D_a|)$ is strictly larger than $H^1$ due to the optimality of Hardy's inequality \eqref{Hardy}. We call the case $a=-(n-2)^2/4$ critical and the case $a>-(n-2)^2/4$ subcritical, respectively. 

Let $\dot \H^s_a=\dot \H^s_a(\R^n)$ be the homogeneous Sobolev space of order $s$ adapted to the operator $H_a$, that is the completion of $C_0^\infty(\R^n)$ with respect to the following pseudo-norm
$$
\norm{f}_{\dot\H^s_a}:=\norm{|D_a|^sf}_{L^2}.
$$
From the above argument, for $a>-(n-2)^2/4$ and $|s|\le1$, $\dot\H^s_a$ is well defined and coincides with the standard space $\dot H^s$. Indeed, this follows from \eqref{2_2} and \eqref{2_3} if $s=1$. The duality argument implies the case $s=-1$. The case $-1<s<1$ follow by interpolating between these two cases. Moreover, the following norm equivalence proved  by \cite{KMVZZ2} plays an essential role in this section. 

%{lemma}
\begin{lemma}[{\cite[Theorem 1.2]{KMVZZ2}}]
\label{lemma_2_1}
For $a\ge -(n-2)^2/4$ we set 
$$\sigma=\sigma(n,a):=1+\sqrt{\frac{(n-2)^2}{4}+a}.$$
Suppose that
$
0<s<\min\{n/2,\ 2,\ \sigma\}
$. Then there exist constants $C_2>C_2>0$ such that
$$
C_1\norm{f}_{\dot H^s}\le \norm{|D_a|^sf}\le C_2\norm{f}_{\dot H^s},\quad f\in C_0^\infty(\R^n). 
$$
\end{lemma}
By the same argument as above, Lemma \ref{lemma_2_1} implies that, for $|s|<\min\{n/2,\ 2,\ \sigma\}$, $\dot\H^s_a$ is well-defined and coincides with $\dot H^s$. 

Let $e^{-itH_a}$ and $S_a(t)$ be the evolution groups associated with the following Schr\"odinger and wave equations with the inverse-square potential, respectively: 
\begin{align*}
i\partial_t u-H_au&=0,\quad u|_{t=0}=u_0;\\
\partial_t^2 v+H_av&=0,\quad (v,\partial_t v)|_{t=0}=(v_0,v_1).
\end{align*}
where, in the matrix form, $S_a(t)$ is given by 
\begin{align}
\label{S_a(t)}
S_a(t)=\begin{pmatrix}\cos(t|D_a|)&|D_a|^{-1}\sin(t|D_a|)\\-|D_a|\sin(t|D_a|)&\cos(t|D_a|)\end{pmatrix}.
\end{align}
Note that these coincide with the free evolutions if $a=0$: $e^{-itH_0}=e^{it\Delta}$ and 
$$
S_0(t)=\begin{pmatrix}\cos(t|D|)&|D|^{-1}\sin(t|D|)\\-|D|\sin(t|D|)&\cos(t|D|)\end{pmatrix},\quad |D|=(-\Delta)^{1/2}.
$$
Since $e^{-itH_a}$ commutes with $H_a$, the spectral theorem implies that $e^{-itH_a}$ also commutes with $\varphi(H_a)$ for any $\varphi\in L^2_{\loc}(\R)$. In particular, choosing $\varphi(\lambda)=|\lambda|^{s/2}$, we see that $e^{-itH_a}$ extends to a unitary on $\dot \H^s_a$ if $|s|<\min\{n/2,\ 2,\ \sigma\}$. This property and Lemma \ref{lemma_2_1} imply that for any $|s|<\min\{n/2,\ 2,\ \sigma\}$, there exists $C>1$ independent of $t$ such that
\begin{align}
\label{2_5}
C^{-1}\norm{f}_{\dot H^s}\le \norm{e^{-itH_a}f}_{\dot H^s}\le C\norm{f}_{\dot H^s}. 
\end{align}
By the same argument, if $|s|<\min\{n/2,\ 2,\ \sigma\}$ and $s-1>-\min\{n/2,\ 2,\ \sigma\}$ then $S_a(t)$ extends to a unitary on $\dot {\mathcal H}_a^s\times \dot {\mathcal H}_a^{s-1}$ and satisfies
\begin{align}
\label{2_6}
C^{-1}\norm{\vec f}_{\dot H^s\times \dot H^{s-1}}\le \norm{S_a(t)\vec f}_{\dot H^s\times \dot H^{s-1}}\le C\norm{\vec f}_{\dot H^s\times \dot H^{s-1}}. 
\end{align}

\subsection{Subcritical case}
Now we are ready to state the main result for the subcritical case.
%theorem
\begin{theorem}	
\label{theorem_2_2}
If $a>-(n-2)^2/4$ and 
$
|s|<\min\{n/2,2,\sigma\}
$ then the following statements hold.
\begin{itemize}
\item The following Schr\"odinger-type wave and inverse wave operators defined on $\dot H^s$ exist: 
\begin{align*}
\mathring W_s^\pm(H_a,H_0)&:=\slim_{t\to\pm\infty}e^{itH_a}e^{-itH_0},\\
\mathring W_s^\pm(H_0,H_a)&:=\slim_{t\to\pm\infty}e^{itH_0}e^{-itH_a},
\end{align*}
where $\slim_{t\to\pm\infty}$ denotes the strong  limit on $\dot H^s$. %They satisfy the intertwining property\begin{align}\label{theorem_2_2_1}\mathring W_s^\pm(H_a,H_0)H_0=H_a\mathring W_s^\pm(H_a,H_0),\quad \mathring W_s^\pm(H_0,H_a)H_a=H_0\mathring W_s^\pm(H_0,H_a). \end{align}
In particular, for any $u_0\in \dot H^s$, there exist unique $u_\pm \in \dot H^s$ such that 
\begin{align}
\label{theorem_2_2_1}
\lim_{t\to \pm\infty}\norm{e^{-itH_a}u_0-e^{-itH_0}u_\pm}_{\dot H^s}=0.
\end{align}
\item Assume in addition that $s>1-\min\{n/2,\ 2,\ \sigma\}$. Then the following wave-type  wave and inverse wave operators defined on $\dot H^{s}\times \dot H^{s-1}$ exist: 
\begin{align*}
\mathring \Omega_s^\pm(S_a,S_0)&:=\slim_{t\to\pm\infty}S_a(-t)S_0(t),\\
\mathring \Omega_s^\pm(S_0,S_a)&:=\slim_{t\to\pm\infty}S_0(-t)S_a(t),
\end{align*}
where $\slim_{t\to\pm\infty}$ denotes the strong limit  on $\dot H^{s}\times \dot H^{s-1}$. In particular, for any $\vec v_0\in \dot H^s\times\dot H^{s-1}$, there exist unique $\vec v_\pm\in \dot H^s\times\dot H^{s-1}$ such that 
\begin{align}
\label{theorem_2_2_2}
\lim_{t\to \pm\infty}\norm{S_a(t)\vec v_0-S_0(t)\vec v_\pm}_{\dot  H^s\times\dot H^{s-1}}=0.\end{align}
\end{itemize}
\end{theorem}
Note that, by \eqref{2_5} and \eqref{2_6}, $e^{itH_a}e^{-itH_0}$ (resp. $S_a(-t)S_0(t)$) is bounded on $\dot H^s$ (resp. $\dot H^s\times\dot H^{s-1}$) uniformly in $t\in \R$. Hence the above definitions of  wave operators make sense. 

%remark
\begin{remark}We give a few comments on this theorem.
\begin{itemize}
\item[(1)] For $s=0$, $\mathring W_0^\pm(H_a,H_0)$ and $\mathring W_0^\pm(H_0,H_a)$ are the standard wave and inverse wave operators defined on $L^2$, respectively, whose existence was proved by \cite{BKM}. Hence the novelty of this theorem lies in the case $s\neq0$. For the standard ones, the intertwining property
$$
\varphi(H_a)\mathring W^\pm_0(H_a,H_0)=\mathring W_0^\pm(H_a,H_0)\varphi(H_0)
$$
holds for any $\varphi\in L^2_{\loc}(\R)$ (see \cite{ReSi_III}). This property combined with \eqref{2_5} implies that  $\mathring W_0^\pm(H_a,H_0)$ and $\mathring W^\pm_0(H_0,H_a)$ extend to bounded operators on $\dot H^s$. We however stress that this does not mean in general the existence of $\mathring W_s^\pm(H_a,H_0)$ and $\mathring W^\pm_s(H_0,H_a)$. 
\item[(2)] Since $\sigma>1$ if $a>-(n-2)^2/4$, one can always take $s=1$ in Theorem \ref{theorem_2_2}. Hence we have the scattering to a free solution in the homogeneous energy space for the Schr\"odinger and wave equations associated with $H_a$ which, together with \ref{theorem_A}, leads Corollary \ref{corollary_1_1} . As mentioned above, under the radial symmetry, \eqref{theorem_2_2_2} with $s=1$ was proved by \cite{MSZ} using a different method. Theorem \ref{theorem_2_2} covers a wider range of $s$ without radial symmetry. 
%\item[(3)] When $a=-(n-2)^2/4$, this theorem does not hold in the sense that the inverse wave operators do not exist even for $s=0$. Indeed, if the inverse wave operator exists then, by the above intertwining property, it extends to a unitary from $\H^1_a$ to $$In fact, the asymptotic behavior of the (radial) solutions does not fit into the  framework of neither the short-range scattering nor the standard modified scattering (see Theorem \ref{theorem_2_7} below). 
\end{itemize}
\end{remark}

The following summarizes a few basic properties of the Schr\"odinger-type wave operators. \begin{corollary}\label{corollary_2_4}The Schr\"odinger-type wave operators satisfy the following properties: \begin{itemize}\item[(1)] $\mathring W_s^\pm(H_a,H_0)$ (resp. $\mathring W_s^\pm(H_0,H_a)$) are isometries from $\dot H^s$ to $\dot \H^s_a$ (resp. $\dot \H^s_a$ to $\dot H^s$);\item[(2)] $\mathring W_s^\pm(H_a,H_0)^*=\mathring W_s^\pm(H_0,H_a)$. Moreover, we have \begin{align*}
\mathring W_s^\pm(H_0,H_a)^*\mathring W_s^\pm(H_a,H_0)&=I_{\dot H^s},\\ \mathring W_s^\pm(H_a,H_0)\mathring W_s^\pm(H_a,H_0)^*&=I_{\dot \H^s_a};\end{align*}\item[(3)] For any $\varphi\in L^2_{\loc}(\R)$, the intertwining property holds:\begin{align*}\varphi(H_a)\mathring W_s^\pm(H_a,H_0)&=\mathring W_s^\pm(H_a,H_0)\varphi(H_0),\\\varphi(H_0)\mathring W_s^\pm(H_0,H_a)&=\mathring W_s^\pm(H_0,H_a)\varphi(H_a).\end{align*}\end{itemize}\end{corollary}

Note that similar properties as items (1) and (2) also hold for the wave-type wave operators with $\dot H^s$ and $\dot\H_a^s$ replaced by $\dot H^s\times \dot H^{s-1}$ and $\dot\H_a^s\times \dot \H_a^{s-1}$, respectively. 

The proof of Theorem \ref{theorem_2_2} essentially relies on Lemma \ref{lemma_2_1} and the following Kato-smoothness result proved by \cite{KaYa} for the free case $a=0$ and by \cite{BKM}  for general cases (see also \cite{BPST2}). 
%{lemma}
\begin{lemma}
\label{lemma_2_5}
Let $a>-(n-2)^2/4$. Then $|x|^{-1}$ is $H_a$-smooth in the sense that
$$
\norm{|x|^{-1}e^{-itH_a}u_0}_{L^2(\R^{1+n})}\le C\norm{u_0}_{L^2(\R^n)}.
$$
\end{lemma}

%proof
\begin{proof}[Proof of Theorem \ref{theorem_2_2}]
The proof is decomposed into several steps. 

{\it Step 1}. We first prove the existence of the standard wave and inverse wave operators $\mathring W^\pm_0(H_a,H_0)$, $\mathring W^\pm_0(H_0,H_a)$. We consider the existence of $\mathring W^+_0(H_0,H_a)$ only, otherwise the proofs being similar. Let $W(t)=e^{itH_0}e^{-itH_a}$, $u\in D(H_a)$ and $v\in D(H_0)$. Note that $D(H_a)$ and $D(H_0)$ are dense in $L^2$. Differentiating $\<W(t)u,v\>$ in $t$ and integrating over the interval $(t',t)$ imply 
$$
\<(W(t)-W(t'))u,v\>=-ia\int_{t'}^t\<|x|^{-1}e^{-i\tau H}u,|x|^{-1}e^{-i\tau H_0}v\>d\tau
$$
where $\<f,g\>=\int_{\R^n}f\overline gdx$. It follows from this formula and Lemma \ref{lemma_2_5} with $a=0$ that 
$$
|\<(W(t)-W(t'))u,v\>|\le C\norm{|x|^{-1}e^{-i\tau H}u}_{L^2((t',t)\times \R^n)}\norm{v}_{L^2}
$$
Taking the supremum over $v\in D(H_0)$ with the condition $\norm{v}_{L^2}=1$, we obtain
$$
\norm{(W(t)-W(t'))u}_{L^2}\le C\norm{|x|^{-1}e^{-i\tau H}u}_{L^2((t',t)\times \R^n)}.
$$
Since $\norm{|x|^{-1}e^{-i\tau H}u}_{L^2(\R^n)}\in L^2(\R_t)$ by Lemma \ref{lemma_2_5}, the right hand side converges to zero as $t,t'\to\infty$. Hence $W(t)u$ converges strongly in $L^2$ as $t\to\infty$.  Since $\norm{W(t)}_{\mathbb B(L^2)}=1$ and $D(H)$ is dense in $L^2$, $\mathring W^+_0(H_0,H_a)$ exists by the density argument. Moreover, since $e^{itH_0}$ is unitary on $L^2$, for any $u_0\in L^2$, setting $u_+:=\mathring W^+_0(H_0,H_a)u_0$ we have
\begin{align*}
\norm{e^{-itH_a}u_0-e^{-itH_0}u_+}_{L^2}=\norm{e^{itH_0}e^{-itH_a}u_0-u_+}_{L^2}\to 0,\quad t\to\infty.
\end{align*}

{\it Step 2}. We next prove the existence of $\mathring W^\pm_s(H_a,H_0)$ and $\mathring W^\pm_s(H_0,H_a)$ for $s\neq0$. As above, we consider the existence of $\mathring W^+_s(H_0,H_a)$ only. By \eqref{2_5}, $W(t)$ is bounded on $\dot H^s$ uniformly in $t$ for any $|s|<\min\{n/2,2,\sigma\}$. It thus is sufficient to show that, as $t\to\infty$, $W(t)u$ converges strongly in $\dot H^s$ for any $u$  belonging a dense subset of $\dot H^s$, say $u\in C_0^\infty(\R^n)$. We fix two exponents $0<s<s'<\min\{n/2,2,\sigma\}$. H\"older's inequality and Plancherel's theorem then imply
\begin{align*}
\norm{(W(t)-W(t'))u}_{\dot H^s}
&\le \norm{(W(t)-W(t'))u}_{L^2}^{\frac{s'-s}{s'}}\norm{(W(t)-W(t'))u}_{\dot H^{s'}}^{\frac{s}{s'}}\\
&\le C\norm{(W(t)-W(t'))u}_{L^2}^{\frac{s'-s}{s'}}\norm{u}_{\dot H^{s'}}^{\frac{s}{s'}}
\end{align*}
uniformly in $t,t'\in \R$. Similarly, we have
$$
\norm{(W(t)-W(t'))u}_{\dot H^{-s}}\le C\norm{(W(t)-W(t'))u}_{L^2}^{\frac{s'-s}{s'}}\norm{u}_{\dot H^{-s'}}^{\frac{s}{s'}}. 
$$
Since $\norm{(W(t)-W(t'))u}_{L^2}$ converges to zero as $t,t'\to\infty$ by the above Step 1, $W(t)u$ converges strongly in $\dot H^s$ as $t\to\infty$. $\mathring W^+_s(H_0,H_a)$ thus exists for any $|s|<\min\{n/2,2,\sigma\}$. Since $e^{-itH_0}$ is a unitary on $\dot H^s$, we have \eqref{theorem_2_2_1} for $|s|<\min\{n/2,2,\sigma\}$ by the same argument as in case of $s=0$. This completes the proof of the first half of the statement.

{\it Step 3}. We next show that the following wave operators for the half-wave equations, 
\begin{align*}
\mathring W_s^\pm (|D_a|,|D|)&:=\slim_{t\to\pm\infty}e^{it|D_a|}e^{-it|D|},\\
\mathring W_s^\pm (|D|,|D_a|)&:=\slim_{t\to\pm\infty}e^{it|D|}e^{-it|D_a|},
\end{align*}
defined on $\dot H^s$ exist. By the same argument as in Step 2, it suffices to prove the case $s=0$. To this end, we employ a version of the invariance principle of wave operators by \cite[Section 4]{KakoYajima} (see also the original works by \cite{Birman} and \cite{Kato_PJM}). This principle particularly implies that, given two non-negative self-adjoint operators $A$ and $B$, if $A-B$ is decomposed as $A-B=V_1^*V_2$ (in the form sense) such that $V_1$ is $A$-smooth and $V_2$ is $B$-smooth, then the wave operators $\mathring W_0^\pm(\sqrt A,\sqrt B)$ and $\mathring W_0^\pm(\sqrt B,\sqrt A)$ exist. We refer to Appendix \ref{appendix_A} below for more details.  It follows from this fact, the above Step 1 and Lemma \ref{lemma_2_5} that $\mathring W_0^\pm (|D_a|,|D|)$ and $\mathring W_0^\pm (|D|,|D_a|)$ exist. 

{\it Step 4.} We finally prove the existence of $\mathring \Omega_s^\pm(S_a,S_0)$ and $\mathring \Omega_s^\pm(S_0,S_a)$. We again consider the existence of $\mathring \Omega_s^+(S_0,S_a)$ only. Let $\vec v_0=(v_0,v_1)\in \dot H^s\times \dot H^{s-1}$ and $\vec v(t)=(v(t),\partial_t v(t))=S_a(t)\vec v_0$. Then $\vec v(t)\in \dot H^{s}\times \dot H^{s-1}$ by \eqref{2_6}. Moreover, one has
\begin{align*}
v(t)
&=\cos(t|D_a|)v_0+|D_a|^{-1}\sin(t|D_a)v_1\\
&=\frac12e^{-it|D_a|}(v_0+i|D_a|^{-1}v_1)+\frac{1}{2}e^{it|D_a|}(v_0-i|D_a|^{-1}v_1) .
\end{align*}
Since $v_0\pm i|D_a|^{-1}v_1$ belong to $\dot H^s$ by Lemma \ref{lemma_2_1}, the above Step 2 implies
\begin{align}
\nonumber 
v(t)
&=\frac12e^{-it|D|}\mathring W_s^+(v_0+i|D_a|^{-1}v_1)+\frac{1}{2}e^{it|D|}\mathring W_s^-(v_0-i|D_a|^{-1}v_1)+o(1)\\
\label{theorem_2_2_proof_1}
&=\cos(t|D|)v_{+,0}+|D|^{-1}\sin(t|D|)v_{+,1}+o(1)
\end{align}
in $\dot H^s$ as $t\to \infty$, where $\mathring W^\pm_s=\mathring W_s^\pm (|D|,|D_a|)$ and
\begin{align*}
v_{+,0}&=\frac12\Big[(\mathring W_s^++\mathring W_s^-)v_0+i(\mathring W_s^+-\mathring W_s^-)|D_a|^{-1}v_1\Big],\\
v_{+,1}&=\frac12\Big[-i|D|(\mathring W_s^+-\mathring W_s^-)v_0+|D|(\mathring W_s^+-\mathring W_s^-)|D_a|^{-1}v_1\Big].
\end{align*}
Note that $(v_{+,0},v_{+,1})$ belongs to $\dot H^s\times \dot H^{s-1}$ since $\mathring W_s^\pm:\dot H^s\to \dot H^s$ by Step 3. Similarly, we have
\begin{align}
\label{theorem_2_2_proof_2}
\partial_tv(t)=-|D|\sin(t|D|)v_{+,0}+\cos(t|D|)v_{+,1}+o(1)
\end{align}
in $\dot H^{s-1}$ as $t\to \infty$. \eqref{theorem_2_2_proof_1} and \eqref{theorem_2_2_proof_2}  imply \eqref{theorem_2_2_2} with $\vec v_+=(v_{+,0},v_{+,1})$. Moreover, $\mathring \Omega_s^+(S_0,S_a)$ can be defined as the map $\vec v_0\mapsto \vec v_+$. This completes the proof of the second half of the statement. 
\end{proof}

\begin{proof}[Proof of Corollary \ref{corollary_2_4}] The first statement (1) is an immediate consequence of the definition of $\mathring W_s^\pm(H_a,H_0)$ and $\mathring W_s^\pm(H_0,H_a)$. To show the items (2) and (3) in Corollary \ref{corollary_2_4}, we observe that Lemma \ref{lemma_2_1} also implies the equivalence the inhomogeneous Sobolev norms,  namely, $$C_1\norm{f}_{H^s}\le \norm{\<H_a\>^{s/2}f}_{L^2}\le C_2\norm{f}_{H^s} $$ for $|s|<\min\{n/2,\ 2,\ \sigma\}$. Then the same proof as above yields the existence of the following wave and inverse wave operators defined on $H^s$: \begin{align*}W_s^\pm(H_a,H_0)&:=\slim_{t\to\pm\infty}e^{itH_a}e^{-itH_0},\\ W_s^\pm(H_0,H_a)&:=\slim_{t\to\pm\infty}e^{itH_0}e^{-itH_a},\end{align*}where $\slim_{t\to\pm\infty}$ denotes the strong limit in $H^s$. In particular, for any $f\in H^{|s|}$, all of $W_{|s|}^\pm(H_a,H_0)f$, $\mathring W_s^\pm(H_a,H_0)f$ and $W_0^\pm(H_a,H_0)f$ coincide with each other due to the uniqueness of the strong limit. Since $H^{|s|}$ is dense in $L^2$ and in $\dot H^s$, the statements (2) and (3) follow from that of $W_0^\pm(H_a,H_0)$ and $W_0^\pm(H_0,H_a)$ which are well-known (see \cite{ReSi_III}). \end{proof}

%remark
\begin{remark}In the above proof of the existence of wave operators in $\dot H^s$, the norm equivalence $\norm{|D_a|^{s'}f}_{L^2}\sim \norm{f}_{\dot H^{s'}}$ for some $s'>s>0$ has played an essential role. As we mentioned in the introduction, we will give in the next section an alternative approach to ensure existence of wave operators in $\dot H^s$ which requires the norm equivalence only for the same $s$ and a few additional assumptions. Note that it is not known whether only the norm equivalence for the same $s$ (without assuming any additional condition) is sufficient or not, to deduce the existence of the wave operator in $\dot H^s$ from the existence of the standard one in $L^2$. \end{remark}

%subsection
\subsection{Critical case}
Next we consider the critical case $
a=(n-2)^2/4.
$ 
To state the result, we introduce a few notation. Let $P_nf$ denote the spherical mean of $f$ given by
$$
P_nf(r):=\omega_{n-1}^{-1}\int_{\mathbb S^{n-1}}f(r\theta)d\sigma,\quad r>0,
$$
and let $P^\perp_n=I-P_n$, where $\omega_{n-1}=|\mathbb S^{n-1}|$ is the volume of the unit sphere $\mathbb S^{n-1}$. Note that, since the potential $a|x|^{-2}$ is radially symmetric, $P_n$ and $P_n^\perp$ commute with $H_a$ and hence with $\varphi(H_a)$ for any $\varphi\in L^2_{\loc}(\R)$. In particular, both of $P_n$ and $P_n^\perp$ extend to partial isometries on $\dot H^s$ and on $\dot \H^s$. Let $U:L^2_{\mathrm{rad}}(\R^n)\to L^2_{\mathrm{rad}}(\R^2)$ be a unitary, given by
$$
Uf(r)=c_n r^{\frac{n-2}{2}}f,\quad U^*f= c_n^{-1} r^{-\frac{n-2}{2}}f,
$$
where $c_n=\sqrt{\omega_{n-1}/w_1}$ and $L^2_{\mathrm{rad}}(\R^n)=\{f\in L^2(\R^n)\ |\ \text{$f$ is radially symmetric}\}$. 

The main result for the critical case is as follows. 

%theorem
\begin{theorem}	
\label{theorem_2_7}
Let $a=-(n-2)^2/4$ and $-1<s\le1$. The the following statements are satisfied: 
\begin{itemize}
\item The Schr\"odinger-type wave and inverse wave operators defined on $P_n^\perp\dot H^s$ exist, namely the following strong limits on $\dot H^s$ exist: 
\begin{align*}
\mathring W_s^\pm(H_a,H_0;P^\perp_n)&:=\slim_{t\to\pm\infty}e^{itH_a}e^{-itH_0}P^\perp_n,\\
\mathring W_s^\pm(H_0,H_a;P^\perp_n)&:=\slim_{t\to\pm\infty}e^{itH_0}e^{-itH_a}P^\perp_n.
\end{align*}
Moreover, for any $u_0\in \dot H^s$, there exist $u_\pm\in P^\perp_n\dot H^s$ such that
\begin{align}
\label{theorem_2_7_1}
\lim_{t\to\pm\infty}\norm{e^{-itH_a}u_0-U^*e^{it\Delta_{\R^2}}UP_nu_0-e^{-itH_0}u_\pm}_{\dot H^s}=0.
\end{align}
\item Assume in addition that $0< s\le 1$. Then the wave-type wave and inverse wave operators defined on $P^\perp_n\dot H^s\times P_n^\perp\dot H^{s-1}$ also exist: 
\begin{align*}
\mathring \Omega_s^\pm(S_a,S_0;P^\perp_n)&:=\slim_{t\to\pm\infty}S_a(-t)S_0(t)P_n^\perp,\\ 
\mathring \Omega_s^\pm(S_0,S_a;P^\perp_n)&:=\slim_{t\to\pm\infty}S_0(-t)S_a(t)P_n^\perp,
\end{align*}
Moreover, for any $\vec v_0\in \dot H^s\times \dot H^{s-1}$, there exist $\vec v_\pm\in  P^\perp_n\dot H^s\times P_n^\perp\dot H^{s-1}$ such that
\begin{align}
\label{theorem_2_7_2}
\lim_{t\to\pm\infty}\norm{S_a(t)\vec v_0-U^*S_{0,\R^2}(t)UP_n\vec v_0-S_{0,\R^n}(t)\vec v_\pm}_{\dot H^s\times \dot H^{s-1}}=0,
\end{align}
where $S_{0,\R^n}(t)$ denotes the free wave evolution group in $\R^n$ and we have used an abuse of notation that $A\vec f=(Af_0,Af_1)$ for $A=U,U^*,P_n$ and $\vec f=(f_0,f_1)$. 
\end{itemize}
\end{theorem}

To prove this theorem, we prepare a couple of lemmas. The following lemma concerns with a few basic properties of $H_a$ in the critical case. 

%{lemma}
\begin{lemma}
\label{lemma_2_8}
Let $a=-(n-2)^2/4$. Then the following statements are satisfied: 
\begin{itemize}
\item %$U$ is an isometry from $\dot\H^1_{\mathrm{rad}}(\R^n)$ to $\dot H^1_{\mathrm{rad}}(\R^2)$. In particular, the form domain $\H^1_{\mathrm{rad}}(\R^n)$ of $P_nH_a$ coincides with $U^*H^1_{\mathrm{rad}}(\R^2)$. Moreover, 
For any radially symmetric $u\in C_0^\infty(\R^n\setminus\{0\})$, one has
\begin{align}
\label{lemma_2_8_1}
H_au=-U^*\Delta_{\R^2}Uu.
\end{align}
\item For $|s|\le1$, we have the following norm equivalence 
\begin{align}
\label{lemma_2_8_2}
C_1\norm{u}_{\dot H^s}\le\norm{|D_a|^{s/2}u}_{L^2}\le C_2\norm{u}_{\dot H^s},\quad u\in P_n^\perp \dot H^s. 
\end{align}
\end{itemize}
\end{lemma}

%proof
\begin{proof}
A direct computation yields that 
$$
-U^*\Delta_{\R^2}UP_n=-\frac{d^2}{dr^2}-\frac{n-1}{r}\frac{d}{dr}-\frac{(n-2)^2}{4r^2}=H_aP_n
$$
in the polar coordinate and \eqref{lemma_2_8_1} follows. 

Next, \eqref{lemma_2_8_2} with $s=1$ is a consequence of the following Hardy-type inequality: 
\begin{align}
\label{lemma_2_8_proof_2}
\frac{n^2}{4}\int\frac{|u|^2}{|x|^2}dx\le \int|\nabla u|^2dx,\quad u\in P_n^\perp C_0^\infty(\R^n).
\end{align}
A simple proof of \eqref{lemma_2_8_proof_2} can be found in \cite[Lemma 2.4]{EkFr} which we record here for reader's convenience. For any $u\in C_0^\infty(\R^n)$, setting $f=Uu$, we have
\begin{align}
\nonumber
\int \left(|\nabla u|^2-\frac{(n-2)^2}{4|x|^2}|u|^2\right)dx&=c_n^{-2}\int |\nabla f|^2|x|^{-n+2}dx\\
\nonumber
&=c_n^{-2}\int_{\mathbb S^{n-1}}\int_{\R_+}\left(|\partial_r f|^2+\frac{|\nabla_\theta f|^2}{r^2}\right)rdrd\sigma.\\
\label{lemma_2_8_proof_3}
&\ge c_n^{-2}\int_{\R_+}\left(\int_{\mathbb S^{n-1}}|\nabla_\theta f|^2d\sigma\right)r^{-1}dr.
\end{align}
Now we let $u\in P_n^\perp C_0^\infty(\R^n)$ so that $f$ is orthogonal to constants, which are eigenfunctions associated with the zero eigenvalue of the spherical Laplacian $\Delta_{\mathbb S^{n-1}}$. Since the first non-trivial eigenvalue of $\Delta_{\mathbb S^{n-1}}$ is equal to $n-1$, we have 
$$
\int_{\mathbb S^{n-1}}|\nabla_\theta f|^2d\sigma \ge (n-1)\int_{\mathbb S^{n-1}}|f|^2d\sigma.
$$
Hence the right hand side of \eqref{lemma_2_8_proof_3} is bounded from below by
$$
(n-1)\int\frac{|u|^2}{|x|^2}dx
$$
and \eqref{lemma_2_8_proof_2} follows. For $-1\le s<1$, \eqref{lemma_2_8_2} follows from the duality and interpolation. 
\end{proof}

We also need the following Kato-smoothness result on the range of $P_n^\perp$: 

%{lemma}
\begin{lemma}[{\cite[Proposition 3.4]{Miz_JDE}}]
\label{lemma_2_9}
Let $a=-(n-2)^2/4$. Then $|x|^{-1}P^\perp_n$ is $H_a$-smooth, i.e.,
$$
\norm{|x|^{-1}P^\perp_n e^{-itH_a}u_0}_{L^2(\R^{1+n})}\le C\norm{u_0}_{L^2(\R^n)}.
$$
\end{lemma}

We are now ready to prove Theorem \ref{theorem_2_7}. 

%proof
\begin{proof}[Proof of Theorem \ref{theorem_2_7}]
The proof is decomposed into three steps. 

{\it Step 1}. We first show the existence of $\mathring W_s^\pm(H_a,H_0;P^\perp_n)$ and $\mathring W_s^\pm(H_0,H_a;P^\perp_n)$. As above, we prove the existence of $\mathring W_s^+(H_0,H_a;P^\perp_n)$ only. Since $P_n^\perp$ commutes with $|x|^{-2}$ and $H_0$, we see that $H_aP_n^\perp=H_a^\perp P_n^\perp$ where $H_a^\perp =H_0^\perp-a|x|^{-2}P_n^\perp$. Hence $$
e^{itH_0}e^{-itH_a}P_n^\perp=e^{itH_0^\perp}e^{-itH_a^\perp}P_n^\perp.
$$
Since $|x|^{-1}P_n^\perp$ is $H_0^\perp$-smooth and $H_a^\perp$-smooth by Lemma \ref{lemma_2_9}, the same argument as above shows that $\mathring W_0^+(H_0,H_a;P^\perp_n)=\mathring W_0^+(H_0^\perp,H_a^\perp;P^\perp_n)$ exists. With the norm equivalence \eqref{lemma_2_8_2} at hand, we also obtain the existence of $\mathring W_s^+(H_0,H_a;P^\perp_n)$ for $|s|<1$ by using the same argument as above. However, this argument cannot be applied directly to the case $s=1$ since we do not know if the norm equivalence \eqref{lemma_2_8_2} holds for some $s>1$. Instead, we employ Theorem \ref{theorem_3_4} in the next section. To this end, we shall check that the conditions (H1)--(H5) in the next section are fulfilled. Let $A=H_0^\perp,B=H_a^\perp,U_A(t)=e^{-itH_0^\perp}$ and $U_B(t)=e^{-itH_a^\perp}$. Then (H1) and (H5) follow from \eqref{lemma_2_8_2} with $s=1$. (H3) and (H4) are general facts which hold for any self-adjoint operator $B$ and its unitary group $e^{-itB}$ (see \cite{ReSi_III}). Note that $H_a^\perp$ is purely absolutely continuous, so $P_{\mathrm{ac}}(H_a^\perp)$ is the identity. To verify (H2), we compute
\begin{align*}
(H_0^\perp+1)^{-1}-(H_a^\perp+1)^{-1}=a(H_a^\perp+1)^{-1}P_n^\perp |x|^{-1}\cdot |x|^{-1}(H_0+1)^{-1}P_n^\perp.
\end{align*}
By \eqref{lemma_2_8_2} and Hardy's inequality \eqref{Hardy}, $\<H_0\>^{1/2} (H_a^\perp+1)^{-1/2}P_n^\perp$ and $(H_a^\perp+1)^{-1/2}|x|^{-1}P_n^\perp$ are bounded on $L^2$, so is the operator
$$
\<H_0\>^{1/2}(H_a^\perp+1)^{-1}P_n^\perp|x|^{-1}=\<H_0\>^{1/2} (H_a^\perp+1)^{-1/2}P_n^\perp\cdot (H_a^\perp+1)^{-1/2}|x|^{-1}P_n^\perp.
$$
Moreover, since $|x|^{-1}\in L^p(\R^n)+L^\infty_0(\R^n)$ with some $n/2<p<n$, $|x|^{-1}(H_0+1)^{-1}P_n^\perp$ is compact on $L^2$ (see e.g. \cite[Example 3.2 and Lemma A.2]{Miz_PAMS}), where $L^\infty_0(\R^n)$ is the completion of $C_0^\infty(\R^n)$ with respect to the $L^\infty$-norm.  Therefore, the operator 
$\<H_0\>^{1/2}((H_0^\perp+1)^{-1}-(H_a^\perp+1)^{-1})$
is compact on $L^2$. Hence (H2) holds. Since (H1)--(H5) are fulfilled, we can apply Theorem \ref{theorem_3_4} to ensure the existence of $\mathring W_1^+(H_0,H_a;P^\perp_n)=\mathring W_1^+(H_0^\perp,H_a^\perp;P^\perp_n)$. 

{\it Step 2}. We next show \eqref{theorem_2_7_1}. Let $u_\pm=\mathring W_s^+(H_0,H_a;P^\perp_n) u_0$. Then $u_\pm\in P_n^\perp \dot H^s$ and 
\begin{align*}
\lim_{t\to\pm\infty}\norm{e^{-itH_a}P_n^\perp u_0-e^{-itH_0}u_\pm}_{\dot H^1}=0
\end{align*}
by the above Step 1. On the other hand, \eqref{lemma_2_8_1} and a density argument imply
\begin{align*}
e^{-itH_a}P_n u_0=U^*e^{it\Delta_{\R^2}}UP_nu_0
\end{align*}
and \eqref{theorem_2_7_1} thus follows. 

{\it Step 3}. We finally show the statement for the wave equation. Let $|D_a^\perp|=(H_a^\perp)^{1/2}$. By the same argument as that for the subcritical case, the wave and inverse wave operators $\mathring W_s^\pm(|D_a^\perp|,|D|;P^\perp_n)$ and $\mathring W_s^\pm(|D|,|D_a^\perp|;P^\perp_n)$  for the half-wave equation exist for $|s|<1$. To deal with the case $s=1$, we consider the quadruplet $(A,B,U_A,U_B)=(H_0^\perp,H_a^\perp,e^{-it|D^\perp|},e^{-it|D_a^\perp|})$. Then the conditions (H1), (H2) and (H5) are the same as in Step 1. Moreover, the conditions (H3) and (H4) in the next section also hold (see Lemma \ref{lemma_3_7}) since $f(\lambda)=\sqrt\lambda$ satisfies $f'>0$ on $(0,\infty)$. Therefore, $\mathring W_1^\pm(|D_a^\perp|,|D|;P^\perp_n)$ and $\mathring W_1^\pm(|D|,|D_a^\perp|;P^\perp_n)$ also exist by Theorem \ref{theorem_3_4}. Then, the same argument as for  the subcritical case yields the existence of $\mathring \Omega_s^\pm(S_a,S_{0,\R^n};P^\perp_n)$ and $\mathring \Omega_s^\pm(S_{0,\R^n},S_a;P^\perp_n)$ for $0<s\le1$. Since $S_a(t)P_n=U^*S_{0,\R^2}(t)UP_n$ which follows from \eqref{lemma_2_8_1}, \eqref{theorem_2_7_2} can be obtained by the same argument as in the proof of  \eqref{theorem_2_7_1}. 
\end{proof}

At the end of this section, we point out an interesting relation between \eqref{NLS} in the critical case and the following 2D inhomogeneous NLS equation:\begin{align}\label{INLS}i\partial_t \varphi+\Delta_{\R^2}\varphi&=\lambda c_n^{-p+1}|x|^{-\frac{(n-2)(p-1)}{2}}|\varphi|^{p-1}\varphi;\quad \varphi:\R\times \R^2\to \C.\end{align}Indeed, by virtue of of the unitarily equivalence \eqref{lemma_2_8_1}, $u(t,r)$ is a radial solution to \eqref{NLS} with $a=-(n-2)^2/4$ if and only if $\varphi(t,r)=r^{(n-2)/2}u(t,r)$ is a radial solution to \eqref{INLS} with the initial condition $\varphi(0,r)=r^{(n-2)/2}u_0$ (at least formally). \eqref{NLW} also has a similar relation with a 2D inhomogeneous NLW equation. Recently, the inhomogeneous NLS equations have attracted increasing attention. There are several existing results on the scattering theory (see e.g. \cite{FaGu}), which might be applicable to study the scattering theory for \eqref{NLS} in the critical case $a=-(n-2)^2/4$. However, we do not pursue this topic in the present paper. 

%section
\section{Abstract scattering theory in Sobolev spaces}
\label{section_3}
In this section we provides an alternative approach in an abstract framework to ensure the existence of wave operators on homogeneous Sobolev spaces. This section is a continuation of the author's previous work \cite{Miz_PAMS} in which the case with inhomogeneous Sobolev spaces was considered. We also give a slight improvement of this previous result which in fact play an important role in the proof for the case with homogeneous Sobolev spaces. %The following Subsection \ref{subsection_3_1} is devoted to an abstract theory. Applications to Schrodinger and wave equations with rough potentials are discussed in Subsection \ref{subsection_4}. 

%\subsection{Abstract criterion for Schr\"odinger-type wave operators}
%\label{subsection_3_1}
Let $\H$ be a (separable) Hilbert space equipped with an inner product $\<\cdot,\cdot\>$ and induced norm $\norm{\cdot}$. $\norm{\cdot}$ also denotes the operator norm on $\H$. With a self-adjoint operator $A$ on $\H$, we associate the Sobolev space $\H_A^s:=\<A\>^{-s/2}\H$ of order $s$ with the norm
$$
\norm{u}_{\H_A^s}=\norm{\<A\>^{s/2}u},
$$
where $\<A\>^{s/2}:=(1+|A|^2)^{s/2}$ is defined via the spectral theorem, namely
$$
\<A\>^{s/2}=\int_{\sigma(A)}\<\lambda\>^{s/2}dE_{A}(\lambda). 
$$
Note that $\H_A^2=D(A)$ and $\H_A^1=D(|A|^{1/2})$.%, where $D(A)$ and $D(|A|^{1/2})$ are the domain and form domain of $A$, respectively. 

Let $A,B$ be two self-adjoint operators and $\{U_A(t)\}_{t\in \R}, \{U_B(t)\}_{t\in \R}$ two families of unitary operators on $\H$. Let $s\in \R$ be fixed. Here we impose the following conditions: 
\begin{itemize}
\item[(H1)] $\H_A^s=\H_B^s$ with equivalent norms. Namely, for any $u\in \H_A^s$ and $v\in \H_B^s$, 
\begin{align*}
%\label{eqn_H1_1}
\norm{u}_{\H_B^s}\le C \norm{u}_{\H_A^s},\quad
%\label{eqn_H1_2}
\norm{v}_{\H_A^s}\le C \norm{v}_{\H_B^s}.
\end{align*}
\end{itemize}
Throughout this section, we always assume (H1) and use the notation 
$$
\H^s:=\H_A^s.
$$
\begin{itemize}
\item[(H2)] With some $r\in \R$, $N\in \N$ and $z_0\in \rho(A)\cap \rho(B)$, $$(A-z_0)^{-N}-(B-z_0)^{-N}\in \mathbb B_\infty(\H^r,\H^s).$$
\item[(H3)] For any $t\in \R$, $U_A(t)$ commutes with $A$, and $U_B(t)$ commutes with $B$. 
\item[(H4)] For any $u\in P_{\mathrm{ac}}(B)\H$, $U_B(t)u\to 0$ weakly in $\H$ as $t\to \pm\infty$. 
\end{itemize}

%remark
\begin{remark}We make several comments on these conditions: 
\label{remark_3_1}
\begin{itemize}
%\item[(1)] It is easy  to see that $\mathbb B_\infty(\H^s)\cup \mathbb B_\infty(\H,\H^s)\subset\mathbb B_\infty(\H^{|s|},\H^s)$. 
\item[(1)] Under (H1), (H2) is independent of the choice of $z_0$. Precisely, if (H2) holds for some $z_0\in \rho(A)\cap \rho(B)$ then $(A-z)^{-N}-(B-z)^{-N}\in \mathbb B_\infty(\H^r,\H^s)$ for any $z\in \rho(A)\cap \rho(B)$. This can be seen from the resolvent equation
\begin{align*}
&(A-z)^{-1}-(B-z)^{-1}\\
&=(A-z)^{-1}(A-z_0)\left\{(A-z_0)^{-1}-(B-z_0)^{-1}\right\}(B-z_0)(B-z)^{-1}
\end{align*}
and an induction argument in $N$. 
\item[(2)] In a precise way, we say that $U_A(t)$ commutes with $A$ if $U_A(t)\H_A^2\subset \H^2_A$ and $U_A(t)Au-AU_A(t)=0$ on $\H_A^2$. Under this condition, $U_A(t)$ also commutes with $\varphi(A)$ for any $\varphi\in L^2_{\loc}(\R)$%\footnote{\ By a direct calculation, $U_A(t)$ commutes with $A$ if and only if $U_A(t)$ commutes with its resolvent $(A-z)^{-1}$ for any $z\in \rho(A)$. Then it follows from Stone's formula that $U_A(t)$ commutes with the spectral measure $E_A(\Omega)$. Finally, the spectral decomposition theorem implies that $U_A(t)$ commutes with $\varphi(A)$.}. 
. In particular, 
$
\norm{U_A(t)u}_{\H_A^s}=\norm{u}_{\H^{s}_A}
$ 
for all $u\in \H^{|s|}_A$ and $s\in \R$. Hence, if $s>0$, $U_A(t)|_{\H_A^s}$ is an isometry on $\H_A^s$, while if $s<0$, $U_A(t)$ can be extended to an isometry on $\H_A^s$. We use the same symbol $U_A(t)$ for such a restriction or an extension. 
\item[(3)] A typical choice of $U_A(t)$ and $U_B(t)$ is $U_A(t)=e^{-itf(A)}$ and $U_B(t)=e^{-itf(B)}$ with some real-valued function $f$ in which case (H3) follows from the spectral theorem. Moreover, for a wide class of functions, we have $P_{\mathrm{ac}}(B)= P_{\mathrm{ac}}(f(B))$ in which case (H4) holds%\footnote{\ For any self-adjoint operator $H_a$, $\<e^{-itH_a}u,v\>$ is given by the Fourier transform of the finite complex measure $d\mu(\lambda)=d\<E_H(\lambda)u,v\>$. If $u\in P_{\mathrm{ac}}(H)\H$, then $d\mu$ is absolutely continuous with respect to $dx$ by definition. Hence $d\mu=gdx$ with some $g\in L^1(\R)$ and we have $\<e^{-itH_a}u,v\>\to 0$ as $t\to\pm\infty$ by Riemann-Lebesgue's lemma.}
. To ensure the property $P_{\mathrm{ac}}(B)= P_{\mathrm{ac}}(f(B))$, it is enough to assume that $B$ is non-negative, $f\in C^1((0,\infty))$ and $f'$ is strictly positive (see Lemma \ref{lemma_3_7} for more details). 
\end{itemize}
\end{remark}
Define the wave operators $W^\pm_s(U_A,U_B)$ on $\H^s$ associated with $U_A(t),U_B(t)$ by
$$
W^\pm_s(U_A,U_B):=\slim_{t\to\pm\infty}U_A(t)^*U_B(t)P_{\mathrm{ac}}(B),
$$
where $U_A(t)^*$ is the adjoint of $U_A(t)$ with respect to $\<\cdot,\cdot\>$ and $\ds\slim_{t\to\pm\infty}$ denotes the strong limit in $\H^s$. Note that $W^\pm_0(U_A,U_B)$ coincide with the standard wave operators defined on $\H$. We also note that $W^\pm_s(U_A,U_B)$ are well-defined (if they exist). Indeed, under (H1) and (H3), $U_A(t),U_B(t),P_{\mathrm{ac}}(B)$ as well as their adjoints extend to bounded operators on $\H^s$ with uniform bounds in $t\in \R$ (see Remark \ref{remark_3_1} (2) above), so $U_A(t)^*U_B(t)P_{\mathrm{ac}}(B)$ is uniformly bounded on  $\H^s$ with respect to $t\in \R$. Hence the definition of $W^\pm_s(U_A,U_B)$ makes sense.

We are now ready to state an abstract result on the existence of $W_s^\pm(U_A,U_B)$.

%theorem
\begin{theorem}
\label{theorem_3_2}
Suppose that the above four conditions {\rm (H1)}--{\rm (H4)} are satisfied for a fixed $s\in \R$. Then, if $W^\pm_0(U_A,U_B)$ exists, so does $W_s^\pm(U_A,U_B)$.  
\end{theorem}

This theorem is a slight improvement of the previous result \cite[Theorem 2.2]{Miz_PAMS}. The previous result mainly considered the case $U_A(t)=e^{-itA}$ and $U_B(t)=e^{-itB}$. Hence, in order to apply it to the following pair of dispersive equations
$$
i\partial_tu-f(H_0)u=0,\quad i\partial_tu-f(H_0+V)u=0,
$$
one has to check the equivalence $\H_{f(H_0)}^s=\H_{f(H_0+V)}^s$ and the compactness of $(f(H_0)-z_0)^{-1}-(f(H_0+V)-z_0)^{-1}$. In the present case, choosing $U_A(t)=e^{-itf(A)}$ and $U_B(t)=e^{-itf(B)}$, it is instead enough to check the equivalence $\H_{H_0}^s=\H_{H_0+V}^s$ and the compactness of $(H_0-z_0)^{-1}-(H_0+V-z_0)^{-1}$. This enables us to deal with a wide class of dispersive equations  in a unified way. This is the main advantage of Theorem \ref{theorem_3_2} compared with the previous work. 

%proof
\begin{proof}[Proof of Theorem \ref{theorem_3_2}]
%Although the proof is similar to that of \cite[Theorem 2.2]{Miz_PAMS}, we give its detail for the sake of self-containedness. 
Let $u\in \H^s$ with $\norm{u}=1$ and set $W(t):=U_A(t)^*U_B(t)P_{\mathrm{ac}}(B)$. We shall prove that $\<A\>^{s/2}(W(t)-W(t'))u$ converges to $0$ strongly in $\H$ as $t,t'\to\infty$ which implies the existence of $W_s^+(U_A,U_B)$. The proof for $W_s^-(U_A,U_B)$ is analogous.

Observe first that one can replace $u$ by $\varphi(B)u$ with some $\varphi\in C_0^\infty(\R)$ since $W(t)$ is uniformly bounded on $\H^s$ in $t\in \R$ as seen above and $\{\varphi(B)u\ |\ u\in \H^s,\ \varphi\in C_0^\infty(\R)\}$ is a dense subset of $ \H^s$. We may assume $\supp \varphi\subset[-R,R]$ with some large $R$ and take $\psi \in C_0^\infty(\R)$ so that $\psi\equiv1$ on $[-4R,4R]$.  Using such $\varphi$ and $\psi$, we then decompose $(W(t)-W(t'))\varphi(B)$ as 
\begin{align*}
(W(t)-W(t'))\varphi(B)=\psi(A)(W(t)-W(t'))\varphi(B)+(1-\psi(A))(W(t)-W(t'))\varphi(B).
\end{align*}
The spectral theorem and the support property of $\psi$ imply $
\norm{\<A\>^{s/2}\varphi(A)}\le C R^{s/2}
$ with some $C>0$ which, together with the existence of $W^+_0(A,B)$, implies, as $t,t'\to\infty$, $$\norm{\<A\>^{s/2}\psi(A)(W(t)-W(t'))\varphi(B)u}\le CR^{s/2}\norm{(W(t)-W(t'))\varphi(B)u}\to0.$$ 

Next it remains to show 
\begin{align}
\label{theorem_3_2_proof_1}
\lim_{t,t'\to\infty}\norm{\<A\>^{s/2}(1-\psi(A))(W(t)-W(t'))\varphi(B)u}=0.
\end{align}
To this end, it is enough to prove
\begin{align}
\label{theorem_3_2_proof_2}
\lim_{t\to\infty}\norm{\<A\>^{s/2}(1-\psi(A))\varphi(B)w(t)}=0,
\end{align}
where we set $w(t)=U_B(t)P_{\mathrm{ac}}(B)u$ for short. Indeed, taking  into account (H3), Remark \ref{remark_3_1} (2) and the unitarity of $U_A(t)$, we learn by \eqref{theorem_3_2_proof_2} that
\begin{align*}
\norm{\<A\>^{s/2}(1-\psi(A))W(t)\varphi(B)u}=\norm{\<A\>^{s/2}(1-\psi(A))\varphi(B)U_B(t)P_{\mathrm{ac}}(B)u}\to0
\end{align*}
as $t\to\infty$  and \eqref{theorem_3_2_proof_1} follows. 
 
We shall show \eqref{theorem_3_2_proof_2}. By the Weierstrass approximation theorem, for any $\ep>0$, there exists a polynomial $P$ such that $\norm{\varphi-P^{(N-1)}}_{L^\infty([-2R,2R])}<\ep$, where $P^{(N-1)}$ is the $(N-1)$-th derivative of $P$. Let $\tilde\varphi\in C_0^\infty(\R)$ be such that $\supp\tilde\varphi\subset [-2R,2R]$ and $\tilde\varphi\equiv1$  on $[-R,R]$. Since $\varphi(B)=\varphi(B)\tilde\varphi(B)$ and $(1-\psi(A))\tilde\varphi(B)=(\psi(B)-\psi(A))\tilde \varphi(B)$, we see that
\begin{align*}
&\<A\>^{s/2}(1-\psi(A))\varphi(B)\\
&=\<A\>^{s/2}(\psi(B)-\psi(A))\tilde\varphi(B)[\varphi(B)-P^{(N-1)}(B)]+\<A\>^{s/2}(1-\psi(A))P^{(N-1)}(B)\tilde\varphi(B),
\end{align*}
By virtue of (H1), the support property $\supp\psi\subset[-4R,4R]$ and the following estimate
$$
\norm{\tilde\varphi(B)[\varphi(B)-P^{(N-1)}(B)]}\le C\norm{\varphi-P^{(N-1)}}_{L^\infty([-2R,2R])}<C\ep,
$$
 the first term satisfies, with some $C>0$ independent of $\ep$ and $t$,
\begin{align}
\label{theorem_3_2_proof_3}
\norm{\<A\>^{s/2}(\psi(B)-\psi(A))\tilde\varphi(B)[\varphi(B)-P^{(N-1)}(B)]w(t)}\le C\ep.
\end{align}
It remains to deal with $\<A\>^{s/2}(1-\psi(A))P^{(N-1)}(B)\tilde\varphi(B)w(t)$.  Cauchy's integral formula implies
$$
P^{(N-1)}(\lambda)=\frac{(N-1)!}{2\pi i}\oint_{|z|=3R} P(z)(z-\lambda)^{-N}dz,\quad \lambda\in \supp \tilde\varphi\subset [-2R,2R]. 
$$
Moreover, since $\supp(1-\psi)\subset [-4R,4R]^c$, it follows from Cauchy's theorem that
$$
\oint_{|z|=3R} P(z)(z-\lambda)^{-N}dz=0,\quad\lambda\in \supp (1-\psi).
$$
Hence, by the functional calculus (or the spectral theorem),
\begin{align*}
&\<A\>^{s/2}(1-\psi(A))P^{(N-1)}(B)\tilde\varphi(B)\\
&=-\frac{(N-1)!}{2\pi i}\oint_{|z|=3R} P(z)\<A\>^{s/2}(1-\psi(A))\left((z-A)^{-N}-(z-B)^{-N}\right)\tilde\varphi(B)dz. 
\end{align*}
We set $I(z)=P(z)\<A\>^{s/2}(1-\psi(A))\left((z-A)^{-N}-(z-B)^{-N}\right)\tilde\varphi(B)$ for short. Since $\supp(1-\psi)\subset [-4R,4R]^c$ and $\supp\tilde\varphi\subset [-2R,2R]$, we have by (H1) and the spectral theorem  that
$$
\sup_{|z|=3R}\norm{\<A\>^{s/2}(1-\psi(A)(z-A)^{-N}\tilde\varphi(B)}+\sup_{|z|=3R}\norm{\<A\>^{s/2}(1-\psi(A)(z-B)^{-N}\tilde\varphi(B)}<\infty.
$$
Therefore, $I(z)$ is bounded on $\H$ uniformly in $z$ satisfying $|z|=3R$: 
\begin{align}
\label{theorem_3_2_proof_4}
\sup_{|z|=3R}\norm{I(z)}<\infty.
\end{align}
Hence, setting $\Omega_\ep=\{|z|=3R\}\cap\{|z-3R|<\ep\ \text{or}\ |z+3R|<\ep\}$, there exists $C>0$  independent of $\ep$ and $t$ such that
\begin{align}
\label{theorem_3_2_proof_5}
\bignorm{\int_{\Omega_\ep}I(z)dzw(t)}\le C\ep.
\end{align}
To deal with the integral of $I(z)$  on $\{|z|=3R\}\setminus\Omega_\ep$, we observe from (H1), (H2) and Remark \ref{remark_3_1} (1) that $I(z)\in \mathbb B_\infty(\H)$ for each $z\in \{|z|=3R\}\setminus\Omega_\ep$. Since the integral 
$$
\int_{\{|z|=3R\}\setminus\Omega_\ep}I(z)dz
$$
converges in norm by \eqref{theorem_3_2_proof_4}, it is also compact on $\H$. Therefore, the condition (H4)  implies
\begin{align}
\label{theorem_3_2_proof_6}
\lim_{t\to\infty}\bignorm{\int_{\{|z|=3R\}\setminus\Omega_\ep}I(z)dzw(t)}=0.
\end{align}
Finally, the estimates \eqref{theorem_3_2_proof_3}, \eqref{theorem_3_2_proof_5} and \eqref{theorem_3_2_proof_6} show that, with some $C>0$ independent of $\ep>0$,
$$
\limsup_{t\to\infty}\norm{\<A\>^{s/2}(1-\psi(A))\varphi(B)w(t)}\le C\ep
$$
which implies \eqref{theorem_3_2_proof_2} since $\ep$ is arbitrarily small. This completes the proof. 
\end{proof}

%remark\begin{remark}When $N=1$, one can prove that $\<B\>^{s/2}(1-\psi(A))\varphi(B)$ is compact on $\H$, which also implies \eqref{theorem_3_2_proof_2}, by using the expression $(1-\psi(A))\varphi(B)=(\psi(B)-\psi(A))\varphi(B)$ and Helffer-Sj\"ostrand's formula. This was done in the previous paper \cite{Miz_PAMS}. \end{remark}

We next consider the case of homogeneous Sobolev spaces. Given a self-adjoint operator $T$, we say that $T$ is strictly positive if $\norm{T^{1/2}u}>0$ for any $u\in D(T^{1/2})\setminus\{0\}$. Let $A$ and $B$ be two non-negative self-adjoint operators on $\H$ and $s>0$ such that $A^{s}$ and $B^{s}$ are strictly positive, where $A^{s}$ and $B^{s}$ are defined via the spectral theorem. Then the homogeneous Sobolev space $\dot \H_A^s$ associated with $A$ of order $s$ is defined as the completion of $\H_A^s$ with respect to the norm 
$$
\norm{u}_{\dot \H_A^s}:=\norm{A^{s/2}u}
$$
Under this setting, we impose the following norm equivalence condition: 
\begin{itemize}
\item[(H5)] $\dot \H_A^s=\dot \H_B^s$ with equivalent norms.
\end{itemize}
Under (H5), we set $\dot \H^s:=\dot \H_A^s$. Then we have the following: 

%theorem
\begin{theorem}	
\label{theorem_3_4}
Let $A$ and $B$ be two non-negative self-adjoint operators on $\H$ and $s>0$ such that $A^{s}$ and $B^{s}$ are strictly positive. Suppose that the above five conditions {\rm (H1)}--{\rm (H5)} are fulfilled and that $W^\pm_0(U_A,U_B)$ exist. Then the following strong limits on $\dot \H^s$ also exist:
$$
\mathring W_s^\pm(U_A,U_B):=\slim_{t\to\pm\infty}U_A(t)^*U_B(t)P_{\mathrm{ac}}(B)\quad \text{on}\quad \dot \H^s. 
$$
\end{theorem}

%proof
\begin{proof}
Under (H3) and (H5), $W(t)=U_A(t)^*U_B(t)P_{\mathrm{ac}}(B)$ is bounded on $\dot \H^s$ uniformly in $t\in \R$. Therefore, we may take $u\in \H^s$ without loss of generality since $\H^s$ is dense in $\dot \H^s$ by definition. Taking into account that $\H^s\hookrightarrow \dot \H^s$ and $W(t)u$ converges strongly in $\H^s$ as $t\to\pm\infty$ by Theorem \ref{theorem_3_2}, $W(t)u$ also converges strongly in $\dot \H^s$  as $t\to\pm\infty$. This completes the proof. 
\end{proof}

\begin{remark}
\label{remark_3_5}
As remarked in the introduction, in order to ensure the existence of wave operators on $\H^s$ (resp. $\dot \H^s$), the above abstract criterions require the norm equivalence condition (H1) (resp. (H5)) only up to the same regularity  $\H^s$ (resp. $\dot \H^s$). 
\end{remark}

\section{Application to scattering theory with rough potentials}
\label{section_4}
This section discusses some applications of the above abstract criterions. Since the inhomogeneous case has been already established in \cite[Section 3]{Miz_PAMS}, we only consider an application of Theorem \ref{theorem_3_4} to the Schr\"odinger and wave equations with rough potentials. 

The same notation as in Section \ref{section_3} is also used in this section. We first prepare simple sufficient conditions to ensure the conditions (H2)--(H4).

%{lemma}
\begin{lemma}
\label{lemma_3_6}
Let $A,B$ be self-adjoint operators on $\H$ such that $\H_A^1=\H_B^1$. Suppose that $\<A\>^{-1/2}(B-A)\<A\>^{-1}$ is compact on $\H$. Then {\rm (H2)} holds for $s=1$. \end{lemma}

%proof
\begin{proof}
A direct computation yields 
\begin{align*}
&\<A\>^{\frac12}((A+i)^{-1}-(B+i)^{-1})
=\<A\>^{\frac12}(B+i)^{-1}(B-A)(A+i)^{-1}\\
&=\<A\>^{\frac12}\<B\>^{-\frac12}\cdot \<B\>(B-z)^{-1}\cdot \<B\>^{-\frac12}\<A\>^{\frac12}\cdot \<A\>^{-\frac12}(B-A)\<A\>^{-1}\cdot\<A\>(A+i)^{-1}.
\end{align*}
Then $\<A\>^{-\frac12}(B-A)\<A\>^{-1}\in \mathbb B_\infty(\H)$ by assumption and otherwise are bounded on $L^2$ since $\H_A^1=\H_B^1$. Hence (H2) for $(s,N,r)=(1,1,0)$. 
\end{proof}

Note that if an operator $C$ is $A$-form compact, that is $\<A\>^{-1/2}C\<A\>^{-1/2}\in \mathbb B_\infty(\H)$, then $\<A\>^{-1/2}C\<A\>^{-1}\in \mathbb B_\infty(\H)$, while the converse can fail: $\<\Delta\>^{-1/2}|x|^{-2}\<\Delta\>^{-1}$ is compact on $L^2(\R^n)$ if $n\ge3$, but $|x|^{-2}$ is not $\Delta$-form compact.

%{lemma}
\begin{lemma}
\label{lemma_3_7}
Let $f$ be a real-valued function on $[0,\infty)$ such that $f'>0$ on $(0,\infty)$, $B$ a non-negative self-adjoint operator on $\H$ and $U_B(t)=e^{-itf(B)}$. Then $U_B(t)$ commutes with $B$. Moreover, {\rm (H4)} holds. 
\end{lemma}

%proof
\begin{proof}
By the spectral theorem, $U_B(t)$ commutes with $B$. Moreover, for any $u\in P_{\mathrm{ac}}(f(B))\H$, $U_B(t)u\to0$ weakly in $\H$ as $t\to\infty$ by Radon-Nikodym and Riemann-Lebesgue theorems. Thus it is enough to show $P_{\mathrm{ac}}(B)=P_{\mathrm{ac}}(f(B))$. Let $Y\subset [0,\infty)$ be a null set. Note that $f^{-1}(Y),f(Y)\subset \R$ are also null sets. Hence, for any $u\in P_{\mathrm{ac}}(B)\H$, 
$$
\<E_{f(B)}(Y)u,u\>=\<E_B(f^{-1}(Y))u,u\>=0.
$$
Similarly, if $v\in P_{\mathrm{ac}}(f(B))\H$ then
$$
\<E_{B}(Y)v,v\>=\<E_{f(B)}(f(Y))v,v\>=0.
$$
These two equalities imply $P_{\mathrm{ac}}(B)=P_{\mathrm{ac}}(f(B))$. 
\end{proof}

Now we consider the Schr\"odinger operators
$$
H_0:=-\Delta,\quad H_V:=H_0+V(x),\quad x\in \R^n,\quad n\ge3,
$$
where $V(x)$ is a real-valued potential. We impose one of the following two types of conditions: 

%assumption\begin{assumption}[Form-subordinated potential]\label{assumption_A}Let $n=3$ and $V\in L^1_{\loc}(\R^3)$ satisfy$$\<|V|u,u\>|\le a\norm{\nabla u}^2$$with some $a<1$ for all $u\in H^1(\R^3)$. \ \end{assumption}

%assumption
\begin{assumption}[Repulsive  potential]
\label{assumption_B}
$V\in L^{n/2,\infty}(\R^n)$, $|x| V \in L^{n,\infty}(\R^n)$ and $x \cdot \nabla V \in L^{n/2,\infty}(\R^n)$. Moreover, there exists $\delta>0$ such that  for all $u\in C_0^\infty(\R^n)$, 
\begin{align}
\label{assumption_B_1}
\<(H_0+V)u,u\>\ge \delta\norm{\nabla u}_{L^2},\quad
\<(H_0-V-x\cdot (\nabla V))u,u\>\ge \delta\norm{\nabla u}_{L^2}. 
\end{align}
Here $L^{p,\infty}(\R^n)$ denotes the weak $L^p$-space. 
\end{assumption}

%assumption
\begin{assumption}[Rough potential]
\label{assumption_C}
$V$ belongs to $L^{n/2}(\R^n)$. Moreover, its negative part $V_-(x)=\max\{0,-V(x)\}$ satisfies
$\norm{V_-}_{L^{n/2}(\R^n)}<S_n$, where 
$
S_n:=n(n-2)|\mathbb S^n|^{2/n}/4
$  
is the best constant in Sobolev's inequality 
\begin{align}
\label{assumption_C_1}
S_n\norm{f}_{L^{\frac{2n}{n-2}}(\R^n)}^2\le \norm{\nabla f}_{L^2(\R^n)}^2.
\end{align}
\end{assumption}

%remark
\begin{remark} %Among these, Assumption \ref{assumption_A} allows the strongest local singularity. For instance, it is enough to assume that $V$ is sufficiently small in the Morrey-Campanato space $M^{n/2,\sigma}(\R^3)$ with $\sigma<n/2$, which is strictly larger than $L^{n/2,\infty}(\R^3)$  (see \cite{FePh}). On the other hand, 
Assumption \ref{assumption_B} covers large scaling-critical potentials having critical singularity at the origin, in all dimensions $n\ge3$. A typical example is the inverse-square potential $a|x|^{-2}$ with $a>-(n-2)^2/4$. Assumption \ref{assumption_B} also covers several potentials satisfying $|x|^2V\notin L^\infty$ (see \cite[a discussion below Example 2.2]{BoMi}). On the other hand, there is no condition on the derivatives of $V$ in Assumption \ref{assumption_C}. Moreover, Assumption \ref{assumption_C} allows the potential $V$ having multiple almost critical local singularities.  
\end{remark}

Under one of Assumptions \ref{assumption_B} and \ref{assumption_C}, $H_0+V$ is non-negative (see \eqref{theorem_3_10_proof_2} and its proof below). Define  $H_V$ as its Friedrichs extension and consider the following Schr\"odinger and wave equations associated with $H_V$: 
\begin{align*}
i\partial_t u-H_Vu&=0,\quad u|_{t=0}=u_0;\\
\partial_t^2 v+H_Vv&=0,\quad (v,\partial_t v)|_{t=0}=(v_0,v_1).
\end{align*}
Let $|D_V|=H_V^{1/2}$ and $S_V(t)$ be the evolution group for the wave equation associated with $H_V$:
$$
S_V(t)=\begin{pmatrix}\cos(t|D_V|)&|D_V|^{-1}\sin(t|D_V|)\\-|D_V|\sin(t|D_V|)&\cos(t|D_V|)\end{pmatrix}.
$$
Then we have the following:

%theorem
\begin{theorem}	
\label{theorem_3_10}
Let $n\ge3$ and $V$ be a real-valued function satisfying one of Assumption \ref{assumption_B} or Assumption \ref{assumption_C}. Then $\mathring W^\pm_s(H_V,H_0)$ and $\mathring W^\pm_s(H_0,H_V)$ defined on $\dot H^s$ exist for all $-1<s\le1$. Moreover, if $0\le s\le1$, $\mathring\Omega^\pm_s(S_V,S_0)$ and $\mathring\Omega^\pm_s(S_0,S_V)$  defined on $\dot H^s\times \dot H^{s-1}$ also exist. 
\end{theorem}

As in the previous section, the following Kato-smoothness result plays a crucial role in the proof of the theorem. 

%{lemma}
\begin{lemma}
\label{lemma_3_11}
Let $n\ge3$. Under Assumption \ref{assumption_B} or \ref{assumption_C}, any functions in $L^{n/2,\infty}(\R^n)$ are both $H_0$-smooth and $H_V$-smooth. %Moreover, if $n=3$ then $|V|^{1/2}$ is  both $H_0$-smooth and $H$-smooth under Assumption \ref{assumption_A}. 
\end{lemma}

%proof
\begin{proof}
Under Assumption \ref{assumption_B}, the desired result is due to \cite[Corollary 2.7]{BoMi}. In the case under  Assumption \ref{assumption_C}, the following uniform Sobolev estimate was proved in \cite[Theorem 1.4]{Miz_APDE}): 
$$
\norm{(H_V-z)^{-1}f}_{L^{\frac{2n}{n-2},2}}\le C\norm{f}_{L^{\frac{2n}{n+2},2}},\quad z\in \C\setminus[0,\infty). 
$$
where $L^{p,q}(\R^n)$ denote the Lorentz spaces (see a textbook \cite{Gra} for the definition and properties of Lorentz spaces). Note that $H_V$ is purely absolutely continuous under Assumption \ref{assumption_B} and there is no zero resonance \cite[Lemma 1.3]{Miz_APDE}. Combining  with H\"older's inequalities for Lorentz norms
\begin{align}
\label{Holder}
\norm{fg}_{L^{2}}\le C\norm{f}_{L^{\frac n2,\infty}}\norm{g}_{L^{\frac{2n}{n-2},2}}, \quad
\norm{fg}_{L^{\frac{2n}{n+2},2}}\le C\norm{f}_{L^{\frac n2,\infty}}\norm{g}_{L^{2}}, 
\end{align}
the above uniform Sobolev estimate implies that, for any $w\in L^{n/2,\infty}$, $w(H_V-z)^{-1}w$ is bounded on $L^2$ uniformly in $z\in \C\setminus\R$. By means of the theory of smooth perturbations \cite[Theorem 5.1]{Kato_MathAnn}, we conclude that $w$ is $H_V$-smooth. This completes the proof. %Finally, under Assumption \ref{assumption_A}, it was shown in \cite[Lemma 1]{FKV} that $$\norm{|V|^{1/2}(H_0-z)^{-1}|V|^{1/2}}\le a,\quad z\in \C\setminus[0,\infty),$$where $\norm{\cdot}$ denotes the operator norm on $L^2$. The proof essentially relies on the pointwise bound$$|G_z(x,y)|\le G_0(x,y),\quad z\in \C\setminus[0,\infty),\ x,y\in \R^3,$$for the kernel $G_0(x,y)=(4\pi|x-y|)^{-1}e^{-\sqrt z|x-y|}$ of the free resolvent $(H_0-z)^{-1}$ in three dimensions. Note that this bound holds only for three dimensions. Since $a<1$, the Birman-Schwinger operator $I+|V|^{1/2}(H_0-z)^{-1}|V|^{1/2}$ can be invertible via the Neumann series, satisfying $$||(I+|V|^{1/2}(H_0-z)^{-1}|V|^{1/2})^{-1}||\le (1-a)^{-1}$$Hence, by the second resolvent equation, $|V|^{1/2}(H_V-z)^{-1}|V|^{1/2}$ is written in the form$$|V|^{1/2}(H_V-z)^{-1}|V|^{1/2}$$
\end{proof}
%proof

\begin{proof}[Proof of Theorem \ref{theorem_3_10}] The proof is decomposed into three steps. 

{\it Step 1}. Setting $A=H_0,B=H_V,U_A(t)=e^{-itH_0}$ and $U_B(t)=e^{-itH_V}$, we first check that the conditions (H1)--(H5) are fulfilled under one of Assumption \ref{assumption_B} or Assumption \ref{assumption_C}. It follows from \eqref{Holder} and Sobolev's inequality 
\begin{align}
\label{theorem_3_10_proof_1}
\norm{f}_{L^{\frac{2n}{n-2},2}(\R^n)}\le C\norm{\nabla f}_{L^2}
\end{align}
that
$$
\int |V||f|^2dx\le C\norm{V}_{L^{\frac n2,\infty}(\R^n)}\norm{f}_{L^{\frac{2n}{n-2},2}(\R^n)}^2\le C\norm{V}_{L^{\frac n2,\infty}(\R^n)}\norm{\nabla f}_{L^2}^2.
$$
Hence, for both cases, $V$ is $H_0$-form bounded. This fact, together with \eqref{assumption_B_1} in case of Assumption \ref{assumption_B} or with the bound $\norm{V_-}_{L^{n/2}(\R^n)}<S_n$ and \eqref{assumption_C_1} in case of Assumption \ref{assumption_C}, yields
\begin{align}
\label{theorem_3_10_proof_2}
C^{-1}\norm{f}_{\dot H^1}\le \norm{|D_V|f}_{L^2}\le C\norm{f}_{\dot H^1},\quad f\in C_0^\infty(\R^n).
\end{align}
Combining with the duality and interpolation arguments, \eqref{theorem_3_10_proof_2} implies   (H1) and (H5) hold for $|s|\le1$ in both cases. (H3) and (H4) are general facts for a unitary group generated by a self-adjoint operator. Under Assumption \ref{assumption_C}, (H2) follows from the $H_0$-form compactness of $V\in L^{n/2}(\R^n)$ and Lemma \ref{lemma_3_6}. Under Assumption \ref{assumption_B}, $V$ is $H_0$-form bounded, but not $H_0$-form compact. However, if we write
$\<H_0\>^{-1/2}V\<H_0\>^{-1}=\<H_0\>^{-1/2}|x|V\cdot |x|^{-1}\<H_0\>^{-1}
$ then $\<H_0\>^{-1/2}|x|V$ is bounded on $L^2$ by \eqref{theorem_3_10_proof_1} and the condition $|x|V\in L^{n,\infty}$ and $|x|^{-1}\<H_0\>^{-1}$ is compact on $L^2$ as seen in the Step 1 of the proof of Theorem \ref{theorem_2_7}. Hence, by Lemma \ref{lemma_3_6}, (H2) also follows under Assumption \ref{assumption_B}. Note that if we choose $A=H_V,B=H_0,U_A(t)=e^{-itH_V}$ and $U_B(t)=e^{-itH_0}$, then (H1)--(H5) can be also verified by the same argument. 

{\it Step 2}. We next show the existence of $\mathring W^\pm_s(H_V,H_0)$ and $\mathring W^\pm_s(H_0,H_V)$. By virtue of Step 1 and Theorem \ref{theorem_3_4}, it sufficient to show the case $s=0$ only. Then, by the same argument as that in the Step 1 of the proof of Theorem \ref{theorem_2_7}, it can be deduced from Lemma \ref{lemma_3_11}  since $|V|^{1/2}\in L^{n,\infty}$   under one of Assumption \ref{assumption_B} or Assumption \ref{assumption_C}. 

{\it Step 3}. By virtue of Lemma \ref{lemma_3_11} and the invariance principle of the wave operators (see Theorem \ref{theorem_appendix_A_2} below), the wave and inverse wave operators for half-wave equation $\mathring W^\pm_s(|D_V|,|D|)$ and $\mathring W^\pm_s(|D|,|D_V|)$ also exist. Therefore, the same argument as in the Step 4 of the proof of Theorem \ref{theorem_2_7} yields the existence of $\Omega^\pm_s(S_V,S_0)$ and $\Omega^\pm_s(S_0,S_V)$. This competes the proof. 
\end{proof}

\section*{Acknowledgments}
The author is partially supported by JSPS KAKENHI Grant-in-Aid for Young Scientists (B) \#JP17K14218 and Grant-in-Aid for Scientific Research (B) \#JP17H02854.

\appendix

\section{The invariance principle of wave operators}
\label{appendix_A}
In this appendix we consider the invariance principle of the wave operators originated by \cite{Birman,Kato_PJM}. The version we record here is due to \cite{KakoYajima}. We only state a particular consequence of \cite[Theorem 4.3]{KakoYajima} which is sufficient for our purpose. The same notation as in Section \ref{section_3} is used in this appendix. Let $A$ be a self-adjoint operator on $\H$, $G$ a densely defined closed operator satisfying $D(A)\subset D(G)$. Recall that $G$ is said to be $A$-smooth if 
$$
\norm{Ge^{-itA}\varphi}_{L^2(\R;\H)}\le C\norm{\varphi},\quad \varphi\in D(A).
$$

%remark
\begin{remark}
\label{remark_appendix_A_1}
It is known (see \cite[Theorem 5.1]{Kato_MathAnn}) that $G$ is $A$-smooth if and only if 
\begin{align*}
\sup_{0<\ep<1}\norm{G(A-\lambda\mp i\ep)^{-1}\varphi}_{L^2(\R_\lambda;\H)}\le C\norm{\varphi}. 
\end{align*}
In particular, $G(A-\lambda-i\ep)^{-1}\varphi$ belongs to the Hardy space $H^2(\Omega_\pm;\H)$, where $\Omega_\pm=\{z\in \C\ |\ 0<\pm \Im z<1\}$. Therefore, for all $\varphi\in \H$, the strong limits 
$$
G(A-\lambda\mp i0)^{-1}\varphi:=\lim_{\ep\searrow0}G(A-\lambda\mp i\ep)^{-1}\varphi\in \H
$$
exist for a.e. $\lambda\in \R$ and satisfy
\begin{align}
\label{remark_appendix_A_1_1}
\norm{G(A-\lambda\mp i0)^{-1}\varphi}_{L^2(\R_\lambda;\H)}\le C\norm{\varphi}. 
\end{align}
\end{remark}

%theorem
\begin{theorem}	
\label{theorem_appendix_A_2}
Let $A,B$ be non-negative self-adjoint operators on $\H$. Suppose there exist densely defined closed operators $V_1,V_2$ with $D(A)\subset D(V_1)$ and $D(B)\subset D(V_2)$ such that $A-B=V_1^*V_2$ in the sense of forms and that $V_1$ is $A$-smooth and $V_2$ is $B$-smooth. Let $f$ be a real-valued function on $[0,\infty)$ such that $f\in C^2(0,\infty)$ and $f'(\lambda)>0$ for all $\lambda>0$. Then the wave operators 
$$
W^\pm(f(A),f(B)):=\slim_{t\to\pm \infty}e^{itf(A)}e^{-itf(B)}
$$
defined on $\H$ exist and coincide with $W^\pm(A,B):=\slim_{t\to\pm \infty}e^{itA}e^{-itB}$. 
\end{theorem}

%remark
\begin{remark}
A typical choice of $f$ is $f(\lambda)=\lambda^{\alpha}$ with $\alpha>0$. In particular, the choice $f(\lambda)=\sqrt\lambda$ corresponds to the scattering for the half-wave equations. 
\end{remark}

%proof
\begin{proof}[Proof of Theorem \ref{theorem_appendix_A_2}]
We consider the case $t\to\infty$ only. The same argument as that in the Step 1 in the proof of Theorem \ref{theorem_2_2} yields the existence of $W^+:=W^+(A,B)$. Then, by the intertwining property, we have $e^{itf(A)}W^+e^{-itf(B)}=W^+$. Hence,
$$
e^{itf(A)}e^{-itf(B)}=e^{itf(A)}(I-W^+)e^{-itf(B)}+W^+.
$$
It remains to show that $(I-W^+)e^{-itf(B)}$ converges to zero strongly in $\H$ as $t\to\infty$. Since $\norm{(I-W^+)e^{-itf(B)}}\le2$, it is enough to prove
\begin{align}
\label{theorem_appendix_A_2_proof_1}
\slim_{t\to\infty}(I-W^+)e^{-itf(B)}E_B(J)=0
\end{align}
with some finite interval $J\Subset (0,\infty)$ by the density argument. Let $u\in D(B)$ and $v\in D(A)$ with $\norm{u}=\norm{v}=1$. In what follows, we let  $L^2_s\H:=L^2((0,\infty)_s;\H)$. It follows from the formula
$$
W^+=I+i\int_0^\infty e^{isA}V_1^*V_2e^{-isB}ds
$$
that
\begin{align*}
|\<(I-W_+)e^{-itf(B)}E_B(J)u,v\>|
&\le \int_0^\infty |\<V_2e^{-isB}e^{-itf(B)}E_B(J)u,V_1e^{-isA}v\>|ds\\
&\le C\norm{V_2e^{-isB}e^{-itf(B)}E_B(J)u}_{L^2_s\H}\norm{V_1e^{-isA}v}_{L^2_s\H}\\
&\le C\norm{V_2e^{-isB}e^{-itf(B)}E_B(J)u}_{L^2_s\H}
\end{align*}
where we have used the $A$-smoothness of $V_1$ in the last line. By the duality, we thus obtain
\begin{align}
\label{theorem_appendix_A_2_proof_2}
\norm{(I-W_+)e^{-itf(B)}E_B(J)u}\le C\norm{V_2e^{-isB}e^{-itf(B)}E_B(J)u}_{L^2_s\H}.
\end{align}
Using the duality we rewrite the right hand side as
\begin{align}
\label{theorem_appendix_A_2_proof_3}
\norm{V_2e^{-isB}e^{-itf(B)}E_B(J)u}_{L^2_s\H}=\sup_{\norm{G}_{L^2_s\H}=1}|\<\!\<V_2e^{-isB}e^{-itf(B)}E_B(J)u,G\>\!\>|,
\end{align}
where $\<\!\<F,G\>\!\>:=\int_0^\infty \<F(s),G(s)\>ds$ is the inner product in $L^2_s\H$. By the density argument, we may assume that $G:(0,\infty)\to \H$ is a simple function and write $G=\sum_{j=1}^N\mathds1_{E_j}(s)g_j$ where $g_j\in \H$ and $E_j\cap E_k=\emptyset$ if $j\neq k$. By the spectral decomposition theorem, Stone's formula and Remark \ref{remark_appendix_A_1}, $\<V_2e^{-isB}e^{-itf(B)}E_B(J)u,g_j\>$ can be written in the form
\begin{align}
\label{theorem_appendix_A_2_proof_4}
\<V_2e^{-isB}e^{-itf(B)}E_B(J)u,g_j\>=\frac{1}{\pi }\int_J e^{-is\lambda}e^{-itf(\lambda)}\<V_2\Im (B-\lambda-i0)^{-1}u,g_j\>d\lambda,
\end{align}
where $\Im (B-\lambda-i0)^{-1}= (2i)^{-1}((B-\lambda-i0)^{-1}-(B-\lambda+i0)^{-1})$. 
Let $$w(\lambda):=\pi^{-1}V_2\Im (B-\lambda-i0)^{-1}u.$$ The $B$-smoothness of $V_2$ and \eqref{remark_appendix_A_1_1} then imply that
$$
\norm{w}_{L^1(J;\H)} \le |J|^{1/2}\norm{w}_{L^2(J;\H)}\le C_J
$$
with some $C_J$ independent of $\ep>0$. By \eqref{theorem_appendix_A_2_proof_4} and Fubini's theorem, we thus can write 
\begin{align}
\nonumber
\<\!\<V_2e^{-isB}e^{-itf(B)}E_J(B)u,G\>\!\>
&=\sum_{j=1}^N\int_{E_j}\int_J e^{-is\lambda}e^{-itf(\lambda)}\<w (\lambda),g_j\>d\lambda ds\\
%&=\lim_{\ep\searrow0}\sum_{j=1}^N\int_{E_j}\int_J e^{-is\lambda}e^{-itf(\lambda)}\<w (\lambda),g_j\>d\lambda ds\\
\label{theorem_appendix_A_2_proof_5}
&=\int_0^\infty \left<\int_Je^{-is\lambda}e^{-itf(\lambda)}w(\lambda)d\lambda,G(s)\right> ds.
\end{align}
Let us set
$$
T_w(t,s):=\int_Je^{-is\lambda}e^{-itf(\lambda)}w(\lambda)d\lambda.
$$
Since $T_t(s)$ is a Fourier transform of $\mathds1_{J}(\lambda)e^{-itf(\lambda)}w(\lambda)$, the Plancherel theorem yields
\begin{align}
\label{theorem_appendix_A_2_proof_6}
\int_0^\infty \norm{T_w(t,s)}^2ds\le \int \norm{w(\lambda)}^2d\lambda\le C.
\end{align}
This, combined with \eqref{theorem_appendix_A_2_proof_2}, \eqref{theorem_appendix_A_2_proof_3} and \eqref{theorem_appendix_A_2_proof_5}, shows
\begin{align}
\label{theorem_appendix_A_2_proof_7}
\norm{(I-W_+)e^{-itf(B)}E_B(J)u}\le C\norm{T_w(t,\cdot)}_{L^2_s\H}. 
\end{align}
It remains to show $\norm{T_w(t,\cdot)}_{L^2_s\H}\to 0$ as $t\to\infty$. Let $\ep>0$ and take a simple function $\tilde G:(0,\infty)\to \H$  of the form $\tilde G(\lambda)=\sum_{j=1}^N\mathds1_{(a_j,b_j)}(\lambda)\tilde g_j$ with $\tilde g_j\in \H$, $(a_j,b_j)\Subset J\Subset (0,\infty)$ such that $\norm{w-\tilde G}_{L^2(\R_\lambda;\H)}<\ep$. Then $\norm{T_{\tilde G}(t,\cdot)}_{L^2_s\H}\le C$ and 
$$
\norm{T_w(t,\cdot)-T_{\tilde G}(t,\cdot)}_{L^2_s\H}\le C\ep 
$$
by \eqref{theorem_appendix_A_2_proof_6}. On the other hand, taking the equality
$$
e^{-is\lambda}e^{-itf(\lambda)}=\frac{i}{s+tf'(\lambda)}\frac{d}{d\lambda}\left(e^{-is\lambda}e^{-itf(\lambda)}\right)
$$
into account, if we set $c_j=\inf_{(a_j,b_j)} f'(\lambda)>0$ then, for each $j=1,...,N$,
$$
\left|\int_{a_j}^{b_j}e^{-is\lambda}e^{-itf(\lambda)}d\lambda\right|\le \frac{1}{s+tf'(b)}+\frac{1}{s+tf'(a)}+\frac{t}{(s+tc_j)^2}\int_{a_j}^{b_j}|f''(\lambda)|d\lambda.
$$
Hence, as $t\to\infty$, 
$$
\norm{T_{\tilde G}(t,s)}_{\H}\le C|t|^{-1}\sum_{j=1}^N\norm{\tilde g_j}\le C|t|^{-1}\to 0
$$
for any $s>0$. By the dominated convergence theorem, $\norm{T_{\tilde G}(t,\cdot)}_{L^2_s\H}\to0$ as $t\to \infty$. Thus, 
$$
\limsup_{t\to \infty}\norm{T_w(t,\cdot)}_{L^2_s\H}\le C\ep. 
$$
Since $\ep>0$ is arbitrarily small, we have  $\norm{T_w(t,\cdot)}_{L^2_s\H}\to 0$ as $t\to\infty$. Together with \eqref{theorem_appendix_A_2_proof_7}, this implies \eqref{theorem_appendix_A_2_proof_1} and thus complete the proof. 
\end{proof}

%%%%%%%%%% Bibliography %%%%%%%%%%%%%%%%%%%

\end{document}